\documentclass[10pt, reqno]{amsart}
\usepackage{amsmath,amssymb, amsthm}
\usepackage[colorlinks,urlcolor=red]{hyperref}
\usepackage{graphicx}
\usepackage{epstopdf}
\usepackage[ruled,vlined]{algorithm2e}
\usepackage{float}
\floatstyle{boxed}
\floatstyle{ruled}
\newfloat{listing}{H}{loi}
\floatname{listing}{Table}
\theoremstyle{plain}
\newtheorem{thm}{Theorem}[section]
\newtheorem{cor}[thm]{Corollary}
\newtheorem{lem}[thm]{Lemma}
\newtheorem{prop}[thm]{Proposition}
\numberwithin{equation}{section} \theoremstyle{definition}
 \newtheorem{defn}[thm]{Definition}
 
 \newtheorem{exmp}[thm]{Example}

\allowdisplaybreaks

\numberwithin{equation}{section}

\renewcommand{\to}{\longrightarrow}

\renewcommand{\to}{\longrightarrow}

\renewcommand{\to}{\longrightarrow}

\begin{document}
\title[Approximating Fixed points of Bregman Generalized $\alpha$-nonexpansive mappings] {Approximating Fixed points of Bregman Generalized $\alpha$-nonexpansive mappings}
\author[K. Muangchoo$^{1,2},$ P. Kumam,
Y.J. Cho and S. Dhompongsa]{ K. Muangchoo$^{1},$ P. Kumam$^{1,*}$, Y.J. Cho$^{2} $ and S. Dhompongsa$^{1} $
}
\date{}
\thanks{$^*$Corresponding author}
\thanks{\it $^\S$ Email address: {\rm poom.kumam@mail.kmutt.ac.th (P. Kumam), yjcho@gsnu.ac.kr (Y.J. Cho), sompong.d@cmu.ac.th (S. Dhompongsa)}} \maketitle

\vspace*{-0.5cm}

\begin{center}
{\footnotesize

$^{1}$KMUTTFixed Point Research Laboratory,  Room SCL 802 Fixed Point Laboratory, \\
 Department of Mathematics, Science Laboratory Building,\\
  Faculty of Science, King Mongkut's University of Technology Thonburi (KMUTT),\\
126 Pracha-Uthit Road, Bang Mod, Thrung Khru, Bangkok 10140, Thailand.\\
$^{2}$Department of Mathematics Education and the RINS\\
Gyeongsang National University, Jinju 660-701, Korea.\\

 }
\end{center}
\maketitle  \hrulefill
\begin{abstract}
In this paper, we introduce a new class of Bregman generalized $\alpha$-nonexpansive mappings in terms of Bregman distances,  and investigate the Ishikawa and Noor iterations for these mappings. We establish weak and strong convergence theorems of Ishikawa and Noor iterative schemes for Bregman generalized $\alpha$-nonexpansive mappings in Banach spaces. Furthermore, we propose an example of our generated mapping and some numerical examples which support our main theorem. Our results are new and improve the recent ones in the literature.\\[5pt]

\end{abstract}
 \vspace{0.5cm}
{\noindent\it Keywords}:  Bergman distances, Bergman function, Bergman-Opial property, Generalized $\alpha$-nonexpansive mappings.\\[5pt]
{\noindent\it{2000 Mathematics Subject Classification:}} 47H09, 47H10, 58C30.
\maketitle

\hrulefill

\section{Introduction}

In 1967, Bregman \cite{Bregman} discovered an effective technique using the so called Bregman distance function $ D_{f} (.,.)$ in the process of designing and analyzing feasibility and optimization algorithms. This opened a growing area of research in which Bregman's technique was applied in various ways in order to design and analyze iterative algorithms for solving not only feasibility and optimization problems, but also algorithms for solving variational inequality problems, equilibrium problems, fixed point problems for nonlinear mappings and so on (see \cite{Bauschke, Butnariu02}).

In recent years, several authors are constructing iterative sequences for finding fixed points of nonlinear mappings by using Bregman distances and the Bregman projection; we refer the readers to \cite{Hussain, Kohsaka01} and the reference therein. In 2003, Bauschke, Borwein and Combettes \cite{Bauschke1, Bauschke2} first introduced the class of Bregman firmly nonexpansive mappings which is a generalization of the classical firmly nonexpansive mappings. A few years before, Reich \cite{Reich} studied the class of Bregman strongly nonexpansive mappings and obtained these common fixed points.


Motivated and inspired by the above-mentioned results. We introduce the classes of Bregman generalized $\alpha$-nonexpansive mappings and investigate the Ishikawa and Noor iterations for these mappings and obtain weak and strong convergence theorems for Bregman generalized $\alpha$-nonexpansive mappings the result in this paper extend and generalize the result of Suzuki \cite{Suzuki} (2008), Pant et al. \cite{Pant} (2017) and Naraghirad et al. \cite{Naraghirad02} (2014).

Throughout this paper, we assume that $ E $ is a real Banach space with $ E^{*} $ as its dual space and norm $ \Vert \cdot \Vert .$  We denote the value of $ x^{*} \in E^{*} $ at $ x \in  E $ by $ \langle x,x^{*} \rangle .$ When $\{x_n\}_{n\in \mathbb{N} }$ is a sequence in $ E, $ we denote the strong convergence and the weak convergence of $\{x_n\}_{n\in \mathbb{N} }$ to $ x \in E $ by $ x_n \rightarrow x $  and $ x_n \rightharpoonup x, $ respectively. Let $ C $ be a nonempty subset of $ E $ and $T:C \to E $ be a map, a point $ x \in C $ is called a fixed point of $ T $ if and only if  $ Tx = x,$ and the set of all fixed points of $ T $ is denoted by $ F(T).$ A mapping $ T $ is said to be
\begin{itemize}
\item[$\bullet$] nonexpansive if
$ \Vert Tx - Ty \Vert \leq \Vert x - y \Vert $ for all $ x, y \in C, $
\vskip 2mm
\item[$\bullet$] quasi-nonexpansive if $ F(T) \neq \emptyset $ and $ \Vert Tx - y \Vert \leq \Vert x - y \Vert $ and  for all  $ x \in C $ and $ y \in F(T),$
\vskip 2mm
\item[$\bullet$] condition \cite{Suzuki} (C which is also known as a Suzuki - type generalized nonexpansive mapping) if \vskip 2mm
 $ \dfrac{1}{2} \| x - Tx \| \leq \| x - y \|  $ implies  
 $ \| Tx - Ty \| \leq \| x - y \| $ for all $ x, y \in C, $  \vskip 2mm 
\item[$\bullet$] $\alpha$-nonexpansive if $ \alpha < 1 $ \vskip 2mm
 $ \| Tx - Ty \|^{2} \leq \alpha \| Tx - y \|^{2} + \alpha \|x - Ty \|^{2} + (1 - 2 \alpha ) \| x - y \|^{2} $ for all $ x, y \in C, $
 \vskip 2mm
 \item[$\bullet$] generalized $\alpha$-nonexpansive \cite{Pant} if there exists an $ \alpha \in [0,1) $ such that \vskip 2mm
 $ \dfrac{1}{2} \| x - Tx \| \leq \| x - y \|  $ implies  
 $ \| Tx - Ty \| \leq \alpha \| Tx - y \| + \alpha \|x - Ty \| + (1 - 2 \alpha ) \| x - y \| $ \vskip 2mm for all $ x, y \in C. $ 
\end{itemize}

The nonexpansivity plays an important role in the study of  the Ishikawa iteration and  the Noor iteration. \\
The Ishikawa iteration \cite{Ishikawa}  given by
		\begin{equation} \label{Ishikawa}
\left \{\begin{array}{l} 

y_n = \beta_n Tx_n + ( 1 - \beta_n ) x_n,\\
x_{n+1}  = \gamma_n Ty_n + ( 1 - \gamma_n) x_n, 
\end{array}\right.
		\end{equation}

where  $\{\beta_n\}_{n\in \mathbb{N}} $ and $\{\gamma_n\}_{n\in \mathbb{N}} $ are  arbitrary sequences in $[0,1.)$ \\

And the Noor iteration \cite{Noor}  given by
		\begin{equation} \label{Noor}
\left \{\begin{array}{l} 

z_n = \alpha_n Tx_n + ( 1 - \alpha_n ) x_n,\\
y_n = \beta_n Tz_n + ( 1 - \beta_n ) x_n,\\
x_{n+1}  = \gamma_n Ty_n + ( 1 - \gamma_n) x_n, 
\end{array}\right.
		\end{equation}
		
where  $\{\alpha_n\}_{n\in \mathbb{N}} $,$\{\beta_n\}_{n\in \mathbb{N}} $ and $\{\gamma_n\}_{n\in \mathbb{N}} $ are  arbitrary sequences in $[0,1)$ satisfying some appropriate conditions. 

The Opial property  is a powerful tool to derive weak or strong convergence of iterative sequences. \cite{Opial}  A Banach space $E$ is said to satisfy the Opial property if the sequence $\{x_n\}_{n\in \mathbb{N}}$ in $E$ converges weakly  to $x \in E,$ then $$\limsup_{n\to\infty} \| x_n-x \| < \limsup_{n \to \infty} \| x_n - y \|  \text{ for all }  y \in E \text{ and }  y \neq x.  $$

In fact, since every weakly convergent sequence is necessarily bounded, we have  $  \displaystyle\limsup_{n\to\infty} \| x_n-x \|  $  and $ \displaystyle\limsup_{n \to \infty} \| x_n - y \|  $ are finite.\\
The Banach spaces $ l^p ( 1 \leq p < \infty )$ satisfy the Opial property, but the  $ L_{p}  [0,2\pi] $ $( 1 \leq p < \infty , p \neq 2)$ spaces do not have.\\

Working with a Bregman distance $ D_{f} $ with respect to $ f $, the following Bregman Opial-like inequality holds for every Banach space $E$: 
$$\limsup_{n\to\infty} D_{f} ( x_n , x ) < \limsup_{n \to \infty} D_{f} ( x_n , y ), $$ 
whenever $ x_{n} \rightharpoonup x \neq y. $ See Lemma \ref{Bregman alpha01} for details. The Bregman-Opial property suggests us to introduce the notions of Bregman generalized $\alpha$-nonexpansive-like mappings.\\

Next, we recall the definition of a Bregman distance which is not a distance in the usual sense. Let $E$ be a Banach space and let $f:E\to \mathbb{R}$ be a strictly convex and {G$\hat{a}$teaux differentiable function.} \emph{The Bregman distance} \cite{Chen} corresponding to $f$ is the function $ D_f : E\times E \rightarrow \mathbb{R} $ defined by
\begin{equation} \label{Bregman}
{ D_f(x , y) = f(x) - f(y) - \langle x - y , \bigtriangledown f( y )\rangle}
\end{equation}
for all $ x , y \in E.$ It is clear that $D_f(x , y) \geq 0 $ for all $ x , y \in E.$ In general, $ D_{f} $ is not symmetric and it does not satisfy the triangle inequality. Clearly, $  D_{f} (x,x) = 0 $, but $  D_{f} (x, y) = 0 $ may not imply $ x = y $ as it happens, for instance, when $ f $ is a linear function on $E.$
 \vskip 2mm
In that case when $E$ is a smooth Banach space, setting $f(x) = \Vert x \Vert ^{2} $ for all $x \in E$, we obtain that $ \nabla f(x) = 2Jx $ for all for all $x \in E.$ Here $J$ is the normalized duality mapping from $E$ into $E^{*}.$ Hence, $D_{f} (x, y) = \phi(x, y)$ as \cite{Pang}

\begin{equation} \label{duality mapping}
{ D_f(x , y) = \phi(x,y) := \Vert x \Vert ^{2} - 2 \langle x , Jy \rangle + \Vert y \Vert ^{2} , \forall x , y \in E.}
\end{equation}\vskip 2mm

If $E$ is a Hilbert space, \eqref{duality mapping} reduces to $ D_{f} (x, y) = \Vert x - y \Vert ^{2} .$ \vskip 2mm

Let $ f: E\to \mathbb{R} $ be a strictly convex and
{G$\hat{a}$teaux differentiable function,} and $C \subseteq E$ be nonempty. Let $T:C\to E $ be a mapping. The fixed point set of $ T $ is  denoted by $ F(T):= \{p \in C : p = Tp \} .$ \vskip 2mm

\begin{itemize}
\item[$\bullet$] $ T $ is said to be Bregman nonexpansive if \vskip 2mm
\begin{center}
$ D_f (Tx , Ty ) \leq D_f ( x , y ),  \forall x, y \in C;$
\end{center}
\vskip 2mm
\item[$\bullet$] $ T $ is said to be Bregman quasi-nonexpansive if $ F(T) \neq \emptyset $ and \vskip 2mm
\begin{center}
 $ D_f ( p , Tx) \leq D_f (p , x), \forall x \in C , \forall p \in F(T);$
\end{center}
\vskip 2mm
\item[$\bullet$] $ T $ is said to be Bregman skew quasi-nonexpansive if $ F(T) \neq \emptyset $ and \vskip 2mm
\begin{center}
 $ D_f ( Tx , p) \leq D_f (x , p), \forall x \in C , \forall p \in F(T);$
\end{center}
\vskip 2mm
\item[$\bullet$] $ T $ is said to be Bregman nonspreading if \vskip 2mm
\begin{center}
$ D_f (Tx , Ty )+ D_f (Ty , Tx) \leq D_f ( Tx , y )+ D_f( Ty, x),  \forall x, y \in C.$
\end{center}
\vskip 2mm
\end{itemize}

In this paper, we propose a new notion of \textit{Bregman generalized $\alpha$-nonexpansive mappings} by using Bregman distances as follow :

A mapping $ T $ is said to be \textit{Bregman generalized $\alpha$-nonexpansive} if $ \alpha \in [0, 1) $ and 
\begin{equation} \label{Bregman alpha}
D_f (Tx , Ty) \leq \alpha D_f (Tx , y) + \alpha D_f (x,Ty) + (1 - 2 \alpha ) D_f (x , y) ,\,\,\, \forall x, y \in C. 
\end{equation}
\vskip 2mm
Let us give an example of a  Bregman generalized $\alpha$-nonexpansive mapping with nonempty fixed point set.

\begin{exmp}
Let $ f : \mathbb{R}\to \mathbb{R} $ be defined by $ f(x) = x^{4}. $ The associated Bregman distance is given by 
$$\aligned
 D_f (x,y) &= x^{4} - y^{4} - (x - y)(4y^{3}) \\
 &= x^{4} + 3y^{4} - 4xy^{3}   \,\,\,\,\,\, \forall x, y \in \mathbb{R}.\\
\endaligned $$
We define a mapping  $ T : [0,0.9] \rightarrow [0,0.9] $ by
$$ Tx = x^2,  \,\,\,\,\,\, \forall x \in [0,0.9].$$

Then $ T $ are not an $\alpha$-nonexpansive mapping and a generalized $\alpha$-nonexpansive mapping, but it is a Bregman generalized $\alpha$-nonexpansive mapping relative to $ D_f $ in the sense of (\ref{Bregman alpha}). We have  $ F(T) = \lbrace 0 \rbrace .$ Plainly, $ T $ is not nonexpansive.

However, $ T $ is Bregman generalized $\alpha$-nonexpansive. Indeed, let $ x \in [0,0.9] $ be fixed. We define mapping $ f : [0,0.9] \rightarrow  [0,0.9] $ by
$$ f(x,y) = \alpha D_f (Tx , y) + \alpha D_f (x,Ty) + (1 - 2 \alpha ) D_f (x , y) -  D_f (Tx , Ty)  , \,\,\,\, \forall y \in [0,0.9], $$ 
where
$$\aligned D_f(Tx,y) &= f(Tx) - f(y) - \langle Tx - y , \nabla f(y) \rangle \\
&= x^8 + 3y^4 - 4 x^2 y^3.\\
D_f(x,Ty) &= f(x) - f(Ty) - \langle x - Ty , \nabla f(Ty) \rangle \\
&= x^4 + 3y^8 - 4 x y^6.\\
D_f(x,y) &= f(x) - f(y) - \langle x - y , \nabla f(y) \rangle \\
&= x^4 + 3y^4 - 4 x y^3.\\
D_f(Tx,Ty) &= f(Tx) - f(Ty) - \langle Tx - Ty , \nabla f(Ty) \rangle \\
&= x^8 + 3y^8 - 4 x^2 y^6.\\
\endaligned $$
Then
$$\aligned 
f(x,y )&=\alpha D_f (Tx , y) + \alpha D_f (x,Ty) + (1 - 2 \alpha ) D_f (x , y) -  D_f (Tx , Ty)\\
&= \alpha (x^8 + 3y^4 - 4 x^2 y^3) + \alpha (x^4 + 3y^8 - 4 x y^6) + (1 - 2 \alpha ) (x^4 + 3y^4 - 4 x y^3) -  (x^8 + 3y^8 - 4 x^2 y^6 )\\
\endaligned $$

Thus we have $ f(x,y) \geq 0,  \forall x, y \in [0,0.9] ,\,\, \alpha \in [\frac{1}{2},1) $ and hence $ T $ is a Bregman generalized $\alpha$-nonexpansive mapping. 
\end{exmp}
The paper is organize as follows. In Section 2, we collect some basic knowledge of Bregman distances. In Section 3, using the Bregman-Opial property, we obtain approximation fixed point theorems. In Sections 4 and 5, we investigate weak and strong convergence of the Ishikawa and Bregman Noor's type iteration for Bregman generalized $\alpha$-nonexpansive mappings. In the last section we show the numerical example.  

\section{PRELIMINARIES}
In this section, we collect several definitions and results, which are used in the following sections. Throughout this paper, let $ E $ be a real Banach space and let $ f:E\to \mathbb{R}$ be a convex function. For any $ x $ in $ E ,$ \emph {the gradient} $\nabla f(x)$ is defined to be the linear functional in $ E ^{*}$ such that
$$ \langle y ,\nabla f(x) \rangle =\lim_{t \rightarrow0}\cfrac{f(x+ty)-f(x)}{t}, \indent \forall y \in E. $$ 
The function $f$ is said to be \emph{G$\hat{a}$teaux differentiable at $x$} if $  \langle y ,\nabla f(x) \rangle\in E^{*}$ for all $ x \in E $. In this case, we denote $ \langle y,\nabla f(x) \rangle$ by $ \nabla f(x)$ is well-defined and $ f $ is \emph{G$\hat{a}$teaux differentiable} if it is \emph{G$\hat{a}$teaux differentiable} everywhere on $E.$ The function $ f $ is also said to be \emph{Fr$\acute{e}$chet differentiable} at $x$ if for all $ \epsilon > 0, $ there exists $ \delta > 0 $ such that $ \Vert y - x \Vert \leq \delta  $ implies
$$ \vert f(y) - f(x) - \langle y - x , \nabla f(x) \rangle | \leq \epsilon \Vert y - x \Vert .$$  
The function $f$ is said to be \emph{Fr$\acute{e}$chet differentiable} if it is \emph{Fr$\acute{e}$chet differentiable} everywhere.
The function $f$ is said to be \emph{convex} on a nonempty subset $E$ of $D_f$ if
\begin{equation}\label{covex}
 f ( \alpha x + ( 1 - \alpha ) y ) \leq \alpha f ( x) + ( 1 - \alpha ) f (y ), 
\end{equation}
for all $x, y \in E, \,\, \alpha \in (0,1).$ It is also said to be \emph{strictly convex} if the strict inequality holds in \eqref{covex} for all $x, y \in dom g $ with $x \neq y$ and $ \alpha \in (0,1).$\\

Let $B$ be the closed unit ball with radius $ r > 0 $ centered at $ 0 \in E $ is denoted by $ rB $ of a Banach space $E.$ A function $ f:E\to \mathbb{R}$ is said to be \emph{strongly coercive} if 
\begin{center}
$ \displaystyle\lim_{\Vert x_{n}\Vert \rightarrow \infty}\cfrac{f(x_{n})}{\Vert x_{n} \Vert} = \infty. $
\end{center} \vskip 2mm

It is also said to be \emph{bounded on bounded sets} or \emph{locally bounded} if $ f(rB) $ is bounded for each $ r > 0.$ Let $ S_E = \{x \in E: \Vert x \Vert = 1 \} $ be the unit sphere of $E.$ 

Let $ E $ is a real Banach space with the norm $ \Vert \cdot \Vert $ and the  dual space $ E^{*}.$ A function $ f:E\to \mathbb{R}$ is said to be \textit{proper} if the set  $ \{ x \in E : f(x)< +\infty \} \neq \emptyset .$
And then a function $ f:E\to \mathbb{R}$ is said to be \emph{uniformly convex on bounded subsets of $E$ or locally uniformly convex on $ E,$} \cite{Zalinescu} if $ \rho_{r}(t) > 0 $ for all $ r, t >0  $, where $\rho_{r} : [0,+\infty ) \rightarrow [0,+\infty ], $ called uniform convexity of $f,$ defined by
 \vskip 2mm
\begin{center}
$ \rho_{r}(t)= \displaystyle\inf_{x,y\in rB,\Vert x - y \Vert = t ,\alpha\in(0,1)} \dfrac{\alpha f(x)+(1-\alpha)f(y)- f(\alpha x + (1 - \alpha )y )}{\alpha(1 - \alpha)}, \,\,\,\, \forall t \geq 0. $ 
\end{center}

It is known $ \rho_{r}(t)$ are nondecreasing function. The function $f$ is also said to be\emph{ locally uniformly smooth on $E$}  \cite{Zalinescu} if the function $ \sigma_{r} : [0,+\infty ) \rightarrow [0,+\infty ],$ defined by \vskip 2mm
 \begin{center}
$ \sigma_{r}(t)= \displaystyle\sup_{x\in rB,y\in S_{E},\alpha\in(0,1)} \dfrac{\alpha f(x+(1-\alpha)ty)+ (1-\alpha)f(x-\alpha t y ) - f(x)}{\alpha(1 - \alpha)}, $
\end{center}

\indent satisfies
$$\lim_{t\downarrow 0}\cfrac{\sigma_{r}(t)}{ t } = 0 , \,\,\,\, \forall r > 0. $$\vskip 2mm
For an uniformly convex on bounded subsets of $E$ map $ f:E\to \mathbb{R},$ we have
\begin{equation} \label{locally}
f(\alpha x + (1 - \alpha ) y ) \leq \alpha f(x) + ( 1 - \alpha )f(y)- \alpha(1 - \alpha) \rho_{r}(\Vert x - y \Vert ), 
\end{equation}

for all $x, y$ in $rB$ and for all $\alpha \in (0, 1).$\\

Let $E$ be a Banach space and let $ f:E\to \mathbb{R}$ a strictly convex and \emph{G$\hat{a}$teaux differentiable function.} By (\ref{Bregman}), the Bregman distance satisfies \cite{Chen}
\begin{equation} \label{Bregmanxyz}
D_{f} (x,z) = D_{f} (x,y)+D_{f} (y,z) + \langle x - y , \nabla f(y) - \nabla f(z)\rangle , \,\,\,\, \forall x,y,z \in E. 
\end{equation}

In particular,
\begin{equation} \label{Bregman commu}
D_{f} ( x,y) = - D_{f} (y,x) + \langle y - x,\nabla f(y) - \nabla f( x )\rangle , \,\,\,\, \forall x,y \in E. 
\end{equation}\vskip 2mm

\begin{lem} \cite{Naraghirad02}. \label{Naraghirad02} 
Let $E$ be a Banach space and $f:E\to \mathbb{R}$ a {G$\hat{a}$teaux differentiable function}  which is uniformly convex on bounded subsets of $E.$ Let $\{x_n\}_{n\in \mathbb{N}}$ and $\{y_n\}_{n\in \mathbb{N}}$ be bounded sequences in $E.$
Then the following assertions are equivalent.\vskip 2mm 
(1) $ \displaystyle \lim_{n\to \infty} D_f(x_n,y_n)=0.$ \vskip 2mm 
(2) $ \displaystyle \lim_{n\to \infty} \|x_n-y_n\parallel=0.$
\end{lem}

The following result was first proved in \cite{Butnariu02} (see also \cite{Kohsaka01}). \vskip 2mm

\begin{lem} \cite{Huang}. \label{Huang}
Let $E$ be a Banach space and let $f:E\to \mathbb{R}$ be a strictly convex and {G$\hat{a}$teaux differentiable function.} Suppose that $\{x_n\}_{n\in \mathbb{N}}$ is a sequence in $E$ such that $x_{n} \rightharpoonup x $ for
some $x \in E.$ Then
$$ \displaystyle \limsup_{n \rightarrow \infty} D_{f} (x_{n}, x )< \displaystyle \limsup_{n \rightarrow \infty} D_{f} (x_{n}, y ), $$ \vskip 2mm

for all $y$ in the interior of the domain of $f$ with $ y \neq x.$
\end{lem}

We call a function $ f:E\to (-\infty , +\infty ] $ \emph{lower semicontinuous}  if $ \{ x \in E : f(x)\leq r \}$ is closed for all $r \in \mathbb{R}. $ For a proper, convex function and lower semicontinuous $f:E\to \mathbb{R},$ the \emph{subdifferential} $ \partial f$ of $f$ is defined by 

$$ \partial f(x) = \{ x^{*} \in E^{*} : f(x) + \langle y - x , x^{*}\rangle \leq f(y),\,\,\,\, \forall y \in E \}, $$ \vskip 2mm 

for all $x \in E.$ It is well known that $\partial f \subset E \times E^{*}$ is maximaly monotone \cite{Rockafellar02}. For any proper convex function and lower semicontinuous $ f:E\to (-\infty , +\infty ], $ \emph{the (Fenchel) conjugate function} $f^{*}$ of $f$ is defined by \vskip 2mm 

\begin{center}
$ f^{*} (x^{*})= \displaystyle \sup_{x \in E} \{ \langle x, x^{*} \rangle - f(x)\},\,\,\,\,\,\forall x^{*} \in E^{*}. $
\end{center}\vskip 2mm

It is well known that
\begin{center}
$ f(x) + f^{*}(x^{*}) \geq \langle x,x^{*}\rangle, \,\,\,\, \forall(x,x^{*}) \in E \times E^{*}.$
\end{center} \vskip 2mm

It is also known that $ ( x,x^{*}) \in \partial f \,\,\, \text {is  equivalent to} $
\begin{equation} \label{conjugate}
f(x) + f^{*}(x^{*}) = \langle x , x^{*} \rangle .
\end{equation}  

We also know that if $ f:E\to (-\infty , +\infty ] $  is a proper  convex function and lower semicontinuous, then $ f^{*}:E^{*} \to (-\infty , +\infty ] $ be a proper convex function and $\text {weak}^{*}$ lower semicontinuous.\vskip 2mm 


\begin{defn}\cite{Kohsaka01} \label{Kohsaka01}
Let $E$ be a Banach space. Then a function $f:E\to \mathbb{R}$ is said to be a \emph {Bregman function} if the following conditions are satisfied:\vskip 2mm
(1) $f$  is continuous, strictly convex and  \emph{G$\hat{a}$teaux differentiable};\vskip 2mm
(2) the set $\{ y \in E : D_{f} (x,y) \leq r \} $  is bounded for all $x$ in $E$ and $r > 0.$ 
\end{defn}\vskip 2mm

The following lemma follows from  Butnariu and Iusem \cite{Butnariu01} and {Z\v{a}linscu} \cite{Zalinescu}:

\begin{lem} \label{Zalinescu}
 Let $E$ be a reflexive Banach space and let $f:E\to \mathbb{R}$ be a strongly coercive Bregman function. Then\vskip 2mm 
(1)  $ \nabla f:E\to E^{*} $ is one-to-one, onto and norm-to-$\text {weak}^{*}$ continuous; \vskip 2mm
(2)  $ \langle x-y, \nabla f(x) - \nabla (y) \rangle = 0 $ if and only if $ x = y; $\vskip 2mm
(3)  $\{ x \in E : D_{f} (x,y) \leq r \} $ is bounded for all $y$ in $E$ and $r > 0;$ \vskip 2mm
(4) dom $ f^{*} = E^{*},\,\, f^{*} $ is {G$\hat{a}$teaux differentiable function} and $ \nabla f^{*}= (\nabla f)^{-1}.$\vskip 2mm
\end{lem}\vskip 2mm

Furthermore, let $E$ be a Banach space and let $C$ is a nonempty, closed convex subset of a reflexive Banach space $E.$ Let $ f:E\to \mathbb{R}$ be a strictly convex and {G$\hat{a}$teaux differentiable function.} Then, we know from \cite{Naraghirad01} that for $x \in E$ and $x_{0} \in C,$ we have

$$D_{f}( x_{0},x)= \displaystyle \min_{y\in C} D_f(y,x) .$$\vskip 2mm

The \emph{Bregman projection} $proj_C^f$ from $E$ onto $C$ is defined by $proj_C^f(x)= x_{0}$ for all $ x\in E.$ It is well known that $ x_{0} = proj_C^f(x) $ if and only if 

\begin{equation} \label{Bregman projection}
\langle y - x_{0}, \nabla f(x) - \nabla f(x_{0}) \rangle \leq 0, \,\,\, \forall y \in C. 
\end{equation}  \vskip 2mm

It is also known that $proj_C^f$ from $E$ onto $C$ has the following property:

\begin{equation} \label{Bregman projection02}
D_{f}(y,proj_C^f(x)) + D_f(proj_C^f(x),x) \leq D_{f}(y,x), \,\,\, \forall y \in C, \forall x \in E;
\end{equation} \vskip 2mm 

see, for instance,  \cite{Butnariu01} for more details.\vskip 2mm
\begin{prop}  \cite{Zalinescu} \label{Zalinescu01}
Let $E$ be a reflexive Banach space and let $f:E\to \mathbb{R}$ be a convex function which is bounded on bounded sets. The following assertions are equivalent.\vskip 2mm 
(1)  $f$ is strongly coercive and uniformly convex on bounded subsets of $E;$\vskip 2mm
(2) dom $ f^{*} = E^{*},\,\, f^{*} $ is bounded on bounded sets and locally  uniformly smooth on $E;$\vskip 2mm
(3) dom $ f^{*} = E^{*},\,\, f^{*} $ is {Fr$\acute{e}$chet differentiable} and $\nabla f^{*}$ is uniformly norm-to-norm continuous on bounded subsets of $E^{*}.$\vskip 2mm
\end{prop}

\begin{prop}  \cite{Zalinescu} \label{Zalinescu02}
Let $E$ be a reflexive Banach space and $f:E\to \mathbb{R}$ a continuous convex function which is strongly coercive. Then the following assertions are equivalent.\vskip 2mm 
(1)  $f$ is bounded on bounded subsets and locally uniformly smooth on $E;$\vskip 2mm
(2) $ f^{*} $ is {Fr$\acute{e}$chet differentiable} and $\nabla f^{*}$ is uniformly norm-to-norm continuous on bounded subsets of $E.$\vskip 2mm
(3) dom $ f^{*} = E^{*},\,\, f^{*} $ is strongly coercive and uniformly convex on bounded subsets of $E.$\vskip 2mm
\end{prop}

\begin{lem} \cite{Butnariu02, Kohsaka01}. \label{Butnariu02}
Let $E $ be a reflexive Banach space, let $f:E\to \mathbb{R}$ be a strongly coercive Bregman function and let $V$ be the function defined by \vskip 2mm
\begin{center}
$ V (x,x^{*}) = f(x) - \langle x,x^{*}\rangle + f^{*}(x^{*}),\,\,\,\,\,\forall x,\in E ,\forall x^{*} \in E^{*}. $
\end{center} \vskip 2mm
The following assertions hold.\vskip 2mm
(1)  $ D_{f} \big(x,\nabla f^{*}(x^{*})\big) = V (x,x^{*}),\,\, \forall x,\in E ,\forall x^{*} \in E^{*}.$ \vskip 2mm
(2)   $ V(x,x^{*}) + \langle\nabla f^{*} (x^{*}) - x,y^{*}\rangle \leq V (x, x^{*} + y^{*}),\,\, \forall x,\in E ,\,\,\forall x^{*} , y^{*} \in E^{*}. $\vskip 2mm
\end{lem}

It also follows from the definition that $V$ is convex in the second variable $x^{*}$ and \vskip 2mm
\begin{center}
$ V \big (x,\nabla f(y) \big) = D_{f} (x,y).$
\end{center}\vskip 2mm

Let $C$ be a nonempty, closed and convex subset of a reflexive Banach space $E,$ and  be a bounded sequence in $E.$ and let $f:E\to \mathbb{R}$ be a lower semicontinuous, strictly convex and \emph{G$\hat{a}$teaux differentiable function.} For any $x \in E,$ we set $$Br(x, \{ x_{n} \})= \displaystyle \limsup_{n \rightarrow \infty} D_{f} (x_{n}, x ).$$ 

The \emph{Bregman asymptotic radius} of $\{x_n\}_{n\in \mathbb{N}}$ relative to $C$ is defined by 
$$Br(C, \{ x_{n} \})= inf \{ Br (x,\{x_{n} \} ) : x \in C \}.$$ 

The \emph{Bregman asymptotic center} of $\{x_n\}_{n\in \mathbb{N}}$ relative to $C$ is defined by 
$$BA(C, \{ x_{n} \})= \{ x \in C : Br (x,\{x_{n} \} )= Br (C, \{ x_{n} \} ) \}.$$ \vskip 2mm

\begin{prop} \label{singleton}
Let $E$ be a reflexive Banach space and $f:E\to \mathbb{R}$ be strictly convex, {G$\hat{a}$teaux differentiable function,} bounded on bounded sets on $E.$ Let $C$ be a nonempty, closed and convex subset of $E.$ If $\{x_n\}_{n\in \mathbb{N}}$ is a bounded sequence of $C$, then $BA(C,\{x_n\}_{n\in \mathbb{N}} )= \{ z \} $ is a singleton.
\end{prop}

\begin{proof}  
In view of the definition of Bregman asymptotic radius, we may assume that $\{x_n\}_{n\in \mathbb{N}}$ converges weakly to $z \in C.$ By Lemma \ref{Huang}, we conclude that $BA(C,\{x_n\}_{n\in \mathbb{N}} )= \{ z \}.$
\end{proof}

Let $S$ be a nonempty set and let $B(S)$ be the Banach space of all bounded real valued functions on $S$ with supremum norm. Let $E$ be a subspace of $B(S)$ and let $\mu$ be and element of $E^*.$ Then, we denote by $\mu(f)$ the value of $\mu$ at $f \in E.$ If $e(s)=1$ for every $s \in S,$ sometimes $\mu(e)$ will be denoted by $\mu(1).$ When $E$ contains constants, a linear functional $\mu$ on $E$ is called a \textit{mean} on $E$ if $\Vert \mu \Vert = \mu (1) = 1.$

\begin{thm} \cite{Takahashi01}
Let $E$ be a subspace of $B(S)$ containing constants and let $\mu$ be  a linear functional on $E.$ Then the following conditions are equivalent:\vskip 2mm 
(1) $\Vert \mu \Vert = \mu (1) = 1,$ i.e., $\mu$ is a \textit{mean} on $E;$\vskip 2mm 
(2) the inequalities $$ \displaystyle \inf_{s \in S} f(s) \leq \mu(f) \leq \displaystyle \sup_{s \in S} f(s) $$ 
hold for each $f \in E.$
\end{thm} 
Let $ l^{\infty}$ be the Banach lattice of bounded real sequences with the supremum norm and let $\mu$ be a linear continuous functional on $ l^{\infty}$ and $x=(x_1,x_2,...) \in {l^{\infty}}.$ Then sometimes. We denote by $\mu_n (x_n)$ the value $\mu(x).$ 
\begin{thm}\cite{Takahashi01} (The existence of Banach limit)  \label{Banach limit}
There exists a linear continuous functional $\mu$ on $ l^{\infty}$ such that $\Vert \mu \Vert = \mu (1) = 1$ and $ \mu ( x_n ) = \mu( x_{n+1} )$ for each $x=(x_1,x_2,...) \in  {l^{\infty}}.$ 
\end{thm}
\begin{flushleft}
\textbf{Observation}\vskip 2mm 
\end{flushleft}

(1)  If $\{x_n\}_{n\in \mathbb{N} \in l^{\infty}}$ and $ x_n \geq 0$ for every $ n \in \mathbb{N} ,$ then $ \mu (x_n) \geq 0 ;$\vskip 2mm 
(2)  If $ x_n = 1 $ for every $ n \in \mathbb{N} $, then $ \mu (x_n) = 1.$\vskip 2mm  
Such a functional $ \mu $ is called \emph{a Banach limit} and the value of $ \mu $ at $\{x_n\}_{n\in \mathbb{N} }\in l^{\infty} $ is denoted by $ \mu_n x_n .$ See, for example \cite{Takahashi01}.\vskip 2mm 
To see some examples of those mappings $T$ satisfying all the stated hypotheses in the following result, we refer the reader to  \cite{Hussain}.

\begin{lem} \cite{Hussain}. \label{Hussain}
Let $E$ be a reflexive Banach space and let $f:E\to \mathbb{R}$ be strictly convex, continuous, strongly coercive, {G$\hat{a}$teaux differentiable function,} and bounded on bounded sets on $E.$ Let $C$ be a nonempty, closed and convex subset of $E.$ Let $T:C\to E $ be a Bregman quasi-nonexpansive mapping. Then $F(T)$ is closed and convex.	
\end{lem}

\begin{lem}  \cite{Mainge} \label{Mainge}
Let $\{a_n\}_{n\in \mathbb{N} }$ be a sequence in $ \mathbb{R} $  with a subsequence $\{a_{n_i}\}_{i \in \mathbb{N} }$ such that $ a_{n_i} < a_{n_i+1} $ for all $ i \in \mathbb{N}.$ Then there exists another subsequence $\{a_{m_k}\}_{k\in \mathbb{N} }$ such that for all (sufficiently large) number $k$ we have
\begin{center}
$ a_{m_k} < a_{m_k+1} $ and $ a_{k} < a_{m_k+1} $.
\end{center}
In fact, we can set $m_k = max \{j \leq k : a_j < a_{j+1} \}.$
\end{lem}

\begin{lem} \cite{Xu} \label{Xu}
Let $\{s_n\}_{n\in \mathbb{N} }$  be a sequence of nonnegative real numbers satisfying
\begin{center}
$ s_{n+1} \leq (1- \gamma_{n} ) s_{n} + \gamma_{n} \delta_{n} ,\,\,\,\, \forall n \geq 1,$
\end{center}
where $\{\gamma_n\}_{n\in \mathbb{N} }$ and $\{\delta_n\}_{n\in \mathbb{N} }$ satisfy the conditions:\vskip 2mm
(1) $\{\gamma_n\}_{n\in \mathbb{N} } \subset [0,1]$ and 
$ \displaystyle \Sigma_{n=1}^\infty \gamma_n = +\infty, $ or equivalently, $ \displaystyle \Pi_{n=1}^\infty (1-\gamma_n) = 0; $\vskip 2mm
(2) $ \displaystyle \limsup_{n\to \infty} \delta_n < 0, $ or \vskip 2mm
(2)' $ \displaystyle \Sigma_{n=1}^\infty \gamma_n \delta_n < \infty. $ \vskip 2mm
Then, $ \displaystyle \lim_{n\to \infty} s_n = 0. $
\end{lem}

\section{Approximating Fixed Points}
In this section, we obtain a fixed point theorem for a generalized $ \alpha $-nonexpansive mapping with respect to the Bregman Opial-like property. 

\begin{lem} \label{Bregman alpha01}
Let $f:E\to \mathbb{R}$ be a strictly convex and {G$\hat{a}$teaux differentiable function.} Let $C$ be a nonempty, closed and convex subset of a reflexive Banach space $E.$ Let $T:C\to E $ be a Bregman generalized $ \alpha $-nonexpansive mapping. Then \vskip 2mm

$\aligned
D_{f}(x, Ty)&\leq D_{f}(x,Tx)+ (1 - \alpha ) D_{f}(x,y) + \alpha D_{f}(Tx, Ty) \\
&\quad+ \alpha \langle x-Tx,\nabla f(y) - \nabla f(Ty) \rangle + \langle x - Tx ,\nabla f(Tx) - \nabla f(Ty) \rangle, \,\,\,\,\,\,\,\, \forall x, y \in C.\\
\endaligned $ \vskip 2mm
\end{lem}
	
\begin{proof}
Let $x, y \in C.$ In view of (\ref{Bregmanxyz}), we have \\\vskip 2mm
$\aligned
D_{f}(Tx, Ty)&\leq \alpha D_{f}(Tx, y) + \alpha D_{f} (x,Ty) + (1 - 2 \alpha) D_{f}(x,y)\\
&=\alpha \big [ D_{f}(Tx, x)+ D_{f}(x,y) + \langle Tx - x ,\nabla f(x) - \nabla f(y) \rangle \big ] \\
&\quad+ \alpha \big [ D_{f}(x, Tx)+ D_{f}(Tx,Ty)+  \langle x - Tx,\nabla f(Tx) - \nabla f(Ty)\rangle ]\\
&\quad+ (1 - 2 \alpha) D_{f}(x,y) \\
&=\alpha D_{f}(Tx, x)+  \alpha D_{f}(x,y) + \alpha \langle Tx - x ,\nabla f(x) - \nabla f(y) \rangle \\
&\quad+ \alpha D_{f}(x, Tx)+ \alpha D_{f}(Tx,Ty)+ \alpha \langle x - Tx,\nabla f(Tx) - \nabla f(Ty)\rangle \\
&\quad+ (1 - 2 \alpha) D_{f}(x,y) \\
&= D_{f}(Tx,x) +(1- \alpha) D_{f}(x,y) + \alpha D_{f}(x, Tx) + \alpha D_{f} (Tx,Ty)\\
&\quad+ \alpha \langle Tx - x ,\nabla f(x) - \nabla f(y) \rangle + \alpha \langle x - Tx,\nabla f(Tx) - \nabla f(Ty)\rangle \\
&= - \alpha D_{f}(x,Tx) + \alpha \langle x-Tx ,\nabla f(x) - \nabla f(Tx) \rangle \\
&\quad+(1- \alpha) D_{f}(x,y) + \alpha D_{f}(x, Tx) + \alpha D_{f} (Tx,Ty)\\
&\quad+ \alpha \langle Tx - x ,\nabla f(x) - \nabla f(y) \rangle + \alpha \langle x - Tx,\nabla f(Tx) - \nabla f(Ty)\rangle \\
&= (1- \alpha) D_{f}(x,y) + \alpha D_{f} (Tx,Ty)\\
&\quad+ \alpha \langle x-Tx ,\nabla f(y) - \nabla f(Ty) \rangle + \alpha \langle x - Tx,\nabla f(Tx) - \nabla f(Ty)\rangle \\
&= (1- \alpha) D_{f}(x,y) + \alpha D_{f} (Tx,Ty)+ \alpha \langle x - Tx,\nabla f(y) - \nabla f(Ty)\rangle \\
\endaligned $ \vskip 2mm
This, together with (\ref{Bregmanxyz}), implies that \vskip 2mm
$\aligned
D_{f}(x, Ty)&= D_{f}(x,Tx) + D_{f} (Tx,Ty) + \langle x - Tx ,\nabla f(Tx) - \nabla f(Ty) \rangle\\
&\leq D_{f}(x,Tx)+ (1 - \alpha ) D_{f}(x,y) + \alpha D_{f}(Tx, Ty) \\
&\quad+ \alpha \langle x-Tx,\nabla f(y) - \nabla f(Ty) \rangle + \langle x - Tx ,\nabla f(Tx) - \nabla f(Ty) \rangle. 
\endaligned$ \\
\end{proof}	

\begin{prop} \label{Bregman alpha02}
(Demiclosedness Principle). Let $f:E\to \mathbb{R}$ be a strictly convex, {G$\hat{a}$teaux differentiable function} and bounded on bounded sets function. Let $C$ be a nonempty subset of a reflexive Banach space $E.$ Let $T:C\to E $ be a Bregman generalized $ \alpha $- nonexpansive mapping. If $ x_{n} \rightharpoonup z $ in $C$ and $ \displaystyle \lim_{n\to \infty} \|Tx_n-x_n\parallel=0$, 
\end{prop}

\begin{proof}
Since $\{x_n\}_{n\in \mathbb{N}}$ converges weakly to $z$ and $ \displaystyle \lim_{n\to \infty} \|Tx_n-x_n\parallel=0,$ both the sequences
$\{x_n\}_{n\in \mathbb{N}}$ and $\{Tx_n\}_{n\in \mathbb{N}}$ are bounded. Since $\nabla f$ is uniformly norm-to-norm continuous on
bounded subsets of $E$ (see, for instance,  \cite{Zalinescu}), we arrive at 
\begin{center}
$ \displaystyle \lim_{n\to \infty} \| \nabla f(x_n)-\nabla f(Tx_n) \parallel=0.$
\end{center}
In view of Lemma \ref{Naraghirad02}, we deduce that $ \displaystyle \lim_{n\to \infty} D_f(x_n,Tx_n)=0$.\\ 
Set \,\,\,\,\,\,\,\,\, $ M_{1}= sup \{ \| \nabla f (x_n) \|, \| \nabla f (Tx_n) \|,\| \nabla f(z) \|, \| \nabla f(Tz) \| : n \in \mathbb{N} \} < + \infty $.\\
By Lemma \ref{Bregman alpha01}, for all $n \in \mathbb{N}$,\vskip 2mm
$\aligned
D_{f}(x_n, Tz)&\leq D_{f}(x_n,Tx_n)+ (1 - \alpha ) D_{f}(x_n,z) + \alpha D_{f}(Tx_n, Tz) \\
&\quad+ \alpha \langle x_n -Tx_n,\nabla f(z) - \nabla f(Tz) \rangle + \langle x_n - Tx_n ,\nabla f(Tx_n) - \nabla f(Tz) \rangle  \\
&=  D_{f}(x_n,Tx_n)  + (1- \alpha )D_{f}(x_n,z)\\
&\quad+ \alpha[ D_{f}(Tx_n,x_n) + D_f(x_n,Tz)+ \langle Tx_n -x_n ,\nabla f (x_n) - \nabla f(Tz)\rangle]\\
&\quad+ \alpha \langle x_n -Tx_n,\nabla f(z) - \nabla f(Tz) \rangle + \langle x_n - Tx_n ,\nabla f(Tx_n) - \nabla f(Tz) \rangle  \\
&=  D_{f}(x_n,Tx_n)  + (1- \alpha )D_{f}(x_n,z)\\
&\quad+ \alpha D_ {f}(Tx_n,x_n) + \alpha D_f (x_n,Tz)+ \alpha \langle Tx_n -x_n ,\nabla f (x_n) - \nabla f (Tz)\rangle\\
&\quad+ \alpha \langle x_n -Tx_n,\nabla f(z) - \nabla f(Tz) \rangle + \langle x_n - Tx_n ,\nabla f(Tx_n) - \nabla f(Tz) \rangle  \\
&=  D_{f}(x_n,Tx_n)  + (1- \alpha )D_{f}(x_n,z)\\
&\quad -\alpha D_ {f}(x_n,Tx_n) + \alpha \langle x_n - Tx_n, \nabla f (x_n)- \nabla f (Tx_n)\rangle \\
&\quad+ \alpha D_f (x_n,Tz)+ \alpha \langle x_n - Tx_n ,\nabla f (Tz) - \nabla f (x_n)\rangle\\
&\quad+ \alpha \langle x_n -Tx_n,\nabla f(z) - \nabla f(Tz) \rangle + \langle x_n - Tx_n ,\nabla f(Tx_n) - \nabla f(Tz) \rangle  \\
&= (1- \alpha )D_f (x_n,Tx_n)+(1- \alpha )D_f (x_n,z)+ \alpha D_f (x_n,T_z)\\
&\quad+ \alpha \langle x_n -Tx_n,\nabla f(z) - \nabla f(Tx_n) \rangle + \langle x_n - Tx_n ,\nabla f(Tx_n) - \nabla f(Tz) \rangle  \\
&\leq (1- \alpha )D_f (x_n,Tx_n)+(1- \alpha )D_f (x_n,z)+ \alpha D_f (x_n,T_z)\\
&\quad+ \alpha\| x_n - Tx_n \| \|\nabla f (z)-\nabla f (Tx_n) \| \\
&\quad+ \| x_{n} - Tx_n \| \|\nabla f(Tx_n) - \nabla f(Tz) \|  \\
&\leq (1- \alpha )D_f (x_n,Tx_n)+(1- \alpha )D_f (x_n,z) + \alpha D_f (x_n,T_z)\\
&\quad+ 2 \alpha M_1 \| x_n - Tx_n \| + 2M_1 \| x_{n} - Tx_n \| \\
&\leq (1- \alpha )D_f (x_n,Tx_n)+ D_f (x_n,z)\\
&\quad+ 2 \alpha M_1 \| x_n - Tx_n \| + 2M_1 \| x_{n} - Tx_n \|. \\
\endaligned $ \vskip 2mm
This implies 
$$ \displaystyle \limsup_{n \rightarrow \infty} D_{f} (x_{n}, Tz )\leq \displaystyle \limsup_{n \rightarrow \infty} D_{f} (x_{n}, z ), $$ \vskip 2mm

From the Bregman Opial-like property, we obtain $Tz = z.$
\end{proof}


To see some examples of those mappings $T$ satisfying all the stated hypotheses in the following result, we refer the reader to \cite{Hussain}.

\begin{thm}\cite{Hussain}.\label{Bregman alpha03}
Let $f:E\to \mathbb{R}$ be a strictly convex, continuous, strongly coercive,{G$\hat{a}$teaux differentiable function,} bounded on bounded sets and uniformly convex on bounded subsets of $E.$ Let $C$ be a nonempty, closed and convex subset of a reflexive Banach space $E.$ Let $T:C\to C$ be a mapping. Let $\{x_n\}_{n\in \mathbb{N} }$  be a bounded sequence of $C$ and let $ \mu $ be a mean on $ l^{\infty} .$ Suppose that
$$ \mu_n D_f(x_n ,Ty )\leq \mu_n D_f (x_n,y), \,\,\, \forall y \in C.$$
It follows from Theorem \ref{Banach limit}. Then $T$ has a fixed point in $C.$
\end{thm}



\begin{cor} \label{Bregman alpha04}
Let $f:E\to \mathbb{R}$ be strictly convex, continuous, strongly coercive, {G$\hat{a}$teaux differentiable function,} bounded on bounded sets and uniformly convex on bounded subsets of $E.$ Let $C$ be a nonempty, bounded, closed and convex subset of a reflexive Banach space $E.$ Let $T:C\to C $ be a Bregman generalized $\alpha$-nonexpansive mapping. Then $T$ has a fixed point.
\end{cor}

\begin{proof}
Let $ \mu $ a Banach limit on $ l^{\infty} $ and $ x \in C $ be such that $\{T^{n}x\}_{n\in \mathbb{N} }$ is bounded. For any $ n \in \mathbb{N} $\\
we have
$$  D_f(T^{n}x ,Ty )\leq \alpha D_f (T^{n}x,y)+ \alpha D_f (T^{n-1}x,Ty)+ (1-2\alpha ) D_f (T^{n-1}x,y), \,\,\, \forall y \in C.$$
Implies that \vskip 2mm
$\aligned
\mu_n D_f(T^{n}x ,Ty )&\leq  \alpha \mu_n D_f (T^{n}x,y)+ \alpha \mu_n D_f (T^{n}x,Ty)+ (1-2\alpha )\mu_n D_f (T^{n}x,y)\\
&\leq (1-\alpha)\mu_n D_f(T^{n}x,y)+ \alpha \mu_n D_f (T^{n}x,Ty).\\
\endaligned $ \vskip 2mm
Thus we have
$$\mu_n D_f(T^{n}x ,Ty )\leq \mu_n D_f (T^{n}x,y), \,\,\, \forall y \in C.$$ \vskip 2mm
It follows from Theorem \ref{Bregman alpha03} that $F(T) \neq \emptyset $.
\end{proof}


\section{Weak and strong convergence theorems for Bregman generalized $\alpha$-nonexpansive mappings}
In this section, we prove weak and strong convergence theorems concerning Bregman generalized $ \alpha$-nonexpansive mappings in a reflexive Banach space.

\begin{lem} \label{Bregman alpha Ishi}
Let $f:E\to \mathbb{R}$ be a strictly convex and {G$\hat{a}$teaux differentiable function.} Let $C$ be a nonempty, closed and convex subset of a reflexive Banach space $E.$ Let $T:C\to C $ be a Bregman skew quasi-nonexpansive mapping with a nonempty fixed point set $F(T).$ Let $\{x_n\}_{n\in \mathbb{N} }$ and $\{y_n\}_{n\in \mathbb{N} }$ be two sequences defined by the Ishikawa iteration \eqref{Ishikawa} 
		\begin{equation*} 
\left \{\begin{array}{l} 

y_n = \beta_n Tx_n + ( 1 - \beta_n ) x_n,\\
x_{n+1}  = \gamma_n Ty_n + ( 1 - \gamma_n) x_n, 
\end{array}\right.
		\end{equation*}\vskip 2mm
		
such that $\{\beta_n\}_{n\in \mathbb{N} }$ and $\{\gamma_n\}_{n\in \mathbb{N} }$ are arbitrary sequences in $[0,1).$ Then the following assertions hold:\\

(1)   $ max \{ D_f(x_{n+1},z), D_f (y_n,z) \} \leq D_f (x_n,z) $ for all $ z$ in $ F(T)$ and $ n=1,2,....$ \vskip 2mm
(2)   $ \displaystyle \lim_{n\to \infty} D_f(x_n,z)$ exists for any $z$ in $F(T).$ \vskip 2mm 
\end{lem} 

\begin{proof}
Let $ z \in F(T). $ In view of \eqref{locally}, we have  
$$\aligned
 D_f (y_n,z) &= D_f  ( \beta_n Tx_n + (1-\beta_n )x_n,z) \\
&\leq \beta_n D_f (Tx_n,z) + (1-\beta_n) D_f(x_n,z)\\
&\leq \beta_n D_f (x_n,z) + (1-\beta_n) D_f(x_n,z)\\ 
&= D_f(x_n ,z). 
 \endaligned $$ Consequently, we get
$$\aligned
D_f (x_{n+1},z) &= D_f (\gamma_n Ty_n + (1-\gamma_n )x_n,z) \\
&\leq \gamma_n D_f (Ty_n,z) + (1-\gamma_n) D_f(x_n,z)\\
&\leq \gamma_n D_f (y_n,z) + (1-\gamma_n) D_f(x_n,z)\\ 
&\leq \gamma_n D_f (x_n,z) + (1-\gamma_n) D_f(x_n,z)\\ 
&= D_f(x_n ,z). 
\endaligned $$ 
This implies that $\{D_f (x_n,z)\}_{n\in \mathbb{N} }$  is a bounded and nonincreasing sequence for all $z$ in $F(T).$\\
Thus we have $ \displaystyle \lim_{n\to \infty} D_f(x_n,z)$  exists for any $z$ in $F(T).$
\end{proof}

\begin{thm} \label{Bregman alpha Ishi02}
Let $f:E\to \mathbb{R}$ be a strictly convex, {G$\hat{a}$teaux differentiable function,} bounded on bounded sets and uniformly convex on bounded subsets of $E.$ Let $C$ be a nonempty, closed and convex subset of a reflexive Banach space $E.$ Let $T:C\to C $ a Bregman generalized $ \alpha$-nonexpansive and Bregman skew quasi-nonexpansive mapping. 
Let $\{\beta_n\}_{n\in \mathbb{N} }$ and $\{\gamma_n\}_{n\in \mathbb{N} }$ be sequences in $[0,1).$ Then $\{ x_n \}_{n\in \mathbb{N} }$ be a sequence with $ x_1 \in C $ defined by the Ishikawa iteration \eqref{Ishikawa}. Assume that\\ $ \displaystyle \lim_{n\to \infty} \|x_n-Tx_n \|=0.$   \\
(a)  If $\{ x_n \}_{n\in \mathbb{N} }$ is bounded and $ \displaystyle \liminf_{n\to \infty} \|Tx_n-x_n\parallel=0,$ then the fixed set $ F(T)\neq \emptyset .$\\
(b) Assume $ F(T)\neq \emptyset.$ Then $ \{ x_n \}_{n\in \mathbb{N} }$ is bounded.
\end{thm}

\begin{proof}
By Corollary \ref{Bregman alpha04}, we see that the fixed point set $F(T)$ of $T$ is nonempty. Assume that $\{ x_n \}_{n\in \mathbb{N} }$ is bounded and $ \displaystyle \liminf_{n\to \infty} \|Tx_n-x_n\parallel=0 .$ Consequently, there is a bounded subsequence $\{ Tx_{n_k} \}_{k\in \mathbb{N} }$ of $\{ Tx_{n} \}_{n\in \mathbb{N} }$ such that $ \displaystyle \lim_{k\to \infty} \Vert Tx_{n_k}-x_{n_k} \Vert =0.$ since $ \nabla g $ is uniformly norm-to-norm continuous on bounded subsets of $E$ (see, for example,  \cite{Zalinescu}), we have

$$ \displaystyle \lim_{k\to \infty} \Vert \nabla f (Tx_{n_k} )- \nabla f (x_{n_k}) \Vert =0.$$ 

In view of Proposition \ref{singleton}, we conclude that $BA(C, \{x_{n_k} \}) = \{z \} $ for some $z$ in $C.$ \\
Let
$$  M_2 = sup \{ \| \nabla f (x_{n_k}) \|, \| \nabla f (Tx_{n_k} ) \|, \| \nabla f (z) \|, \| \nabla f (Tz) \| : k \in \mathbb{N} \} < +\infty. $$
It follows from Lemma \ref{Bregman alpha02} that \\

$\aligned
D_{f}(x_{n_k}, Tz)&\leq D_{f}(x_{n_k},Tx_{n_k})+ (1 - \alpha ) D_{f}(x_{n_k},z) + \alpha D_{f}(Tx_{n_k}, Tz) \\
&\quad+ \alpha \langle x_{n_k} -Tx_{n_k},\nabla f(z) - \nabla f(Tz) \rangle + \langle x_{n_k} - Tx_{n_k} ,\nabla f(Tx_{n_k}) - \nabla f(Tz) \rangle  \\
&=  D_{f}(x_{n_k},Tx_{n_k})  + (1- \alpha )D_{f}(x_{n_k},z)\\
&\quad+ \alpha[ D_{f}(Tx_{n_k},x_{n_k}) + D_f (x_{n_k},Tz)+ \langle Tx_{n_k} -x_{n_k} ,\nabla f (x_{n_k}) - \nabla f (Tz)\rangle]\\
&\quad+ \alpha \langle x_{n_k} -Tx_{n_k},\nabla f(z) - \nabla f(Tz) \rangle + \langle x_{n_k} - Tx_{n_k} ,\nabla f(Tx_{n_k}) - \nabla f(Tz) \rangle  \\
&=  D_{f}(x_{n_k},Tx_{n_k})  + (1- \alpha )D_{f}(x_{n_k},z)\\
&\quad+ \alpha D_ {g}(Tx_{n_k},x_{n_k}) + \alpha D_f (x_{n_k},Tz)+ \alpha \langle Tx_{n_k} -x_{n_k} ,\nabla f (x_{n_k}) - \nabla f (Tz)\rangle\\
&\quad+ \alpha \langle x_{n_k} -Tx_{n_k},\nabla f(z) - \nabla f(Tz) \rangle + \langle x_{n_k} - Tx_{n_k} ,\nabla f(Tx_{n_k}) - \nabla f(Tz) \rangle  \\
&=  D_{f}(x_{n_k},Tx_{n_k})  + (1- \alpha )D_{f}(x_{n_k},z)\\
&\quad -\alpha D_ {f}(x_{n_k},Tx_{n_k}) + \alpha \langle x_{n_k} - Tx_{n_k}, \nabla f (x_{n_k})- \nabla f (Tx_{n_k})\rangle \\
&\quad+ \alpha D_f (x_{n_k},Tz)+ \alpha \langle x_{n_k} - Tx_{n_k} ,\nabla f (Tz) - \nabla f (x_{n_k})\rangle\\
&\quad+ \alpha \langle x_{n_k} -Tx_{n_k},\nabla f(z) - \nabla f(Tz) \rangle + \langle x_{n_k} - Tx_{n_k} ,\nabla f(Tx_{n_k}) - \nabla f(Tz) \rangle  \\
&= (1- \alpha )D_f (x_{n_k},Tx_{n_k})+(1- \alpha )D_f (x_{n_k},z)+ \alpha D_f (x_{n_k},T_z)\\
&\quad+ \alpha \langle x_{n_k} - Tx_{n_k},\nabla f(z) - \nabla f(Tx_{n_k}) \rangle + \langle x_{n_k} - Tx_{n_k} ,\nabla f(Tx_{n_k}) - \nabla f(Tz) \rangle  \\
&\leq (1- \alpha )D_f (x_{n_k},Tx_{n_k})+(1- \alpha )D_f (x_{n_k},z)+ \alpha D_f (x_{n_k},T_z)\\
&\quad+ \alpha\| x_{n_k} - Tx_{n_k} \| \|\nabla f (z)-\nabla f (Tx_{n_k}) \| + \| x_{n_k} - Tx_{n_k} \| \|\nabla f(Tx_{n_k}) - \nabla f(Tz) \|  \\
&\leq (1- \alpha )D_f (x_{n_k},Tx_{n_k})+(1- \alpha )D_f (x_{n_k},z)+ \alpha D_f (x_{n_k},T_z)\\
&\quad+ 2 \alpha M_1 \| x_{n_k} - Tx_{n_k} \| + 2M_1 \| x_{n_k} - Tx_{n_k} \| \\
&\leq (1- \alpha )D_f (x_{n_k},Tx_{n_k})+ D_f (x_{n_k},z)\\
&\quad+ 2 \alpha M_1 \| x_{n_k} - Tx_{n_k} \| + 2M_1 \| x_{n_k} - Tx_{n_k} \|, \,\,\, \text{for} \,\,\, k = 1,2,.... \\
\endaligned $ \vskip 2mm
This implies
$$ \displaystyle \limsup_{n \rightarrow \infty} D_{f} (x_{n_k}, Tz )\leq \displaystyle \limsup_{n \rightarrow \infty} D_{f} (x_{n_k}, z ). $$ \\
From the Bregman Opial-like property, we obtain $Tz = z$.\\

Let $F(T)\neq \emptyset $ and let $z \in F(T)$. It follows from Lemma \ref{Bregman alpha Ishi} that $\displaystyle \lim_{n\to \infty} \|x_n-z \parallel=0,$  exists
and hence $\{ x_n \}_{n\in \mathbb{N} }$ is bounded. This implies that the sequence $\{ Ty_n \}_{n\in \mathbb{N} }$ is bounded too.
\end{proof}

\begin{thm} \label{Bregman alpha Ishi03}
Let $f:E\to \mathbb{R}$ be a strictly convex, {G$\hat{a}$teaux differentiable function,} bounded on bounded sets and uniformly convex on bounded subsets of $E.$ Let $C$ be a nonempty, closed and convex subset of a reflexive Banach space $E.$ Let $T:C\to C $ a Bregman generalized $ \alpha$-nonexpansive and Bregman skew quasi-nonexpansive mapping with $ F(T)\neq \emptyset .$ Let $\{\beta_n\}_{n\in \mathbb{N} }$ and $\{\gamma_n\}_{n\in \mathbb{N} }$ be sequences in $[0,1),$ and let $ \{ x_n \}_{n\in \mathbb{N} }$ be a sequence with $ x_1 \in C $ defined by the Ishikawa iteration \eqref{Ishikawa}. Then $\{ x_n \}_{n\in \mathbb{N} }$ converges weakly to a fixed point of $T.$ 
\end{thm}
\begin{proof}
By Corollary \ref{Bregman alpha04}, we see that the fixed point set $F(T)$ of $T$ is nonempty.It follows from Theorem \ref{Bregman alpha Ishi02} that $\{ x_n \}_{n\in \mathbb{N} }$ is bounded and $ \displaystyle \lim_{n\to \infty} \|Ty_{n}-x_{n} \|=0. $ Since $E$ is reflexive, then there exists a subsequence $\{ x_{n_i} \}_{i\in \mathbb{N} }$  of $\{ x_n \}_{n\in \mathbb{N} }$ such that $ x_{n_i} \rightharpoonup p \in C $ as $ i \rightarrow \infty. $ By Proposition \ref{Bregman alpha02}, $ p \in F(T).$ \\
We claim that $ x_n \rightharpoonup p $ as $ n \rightarrow \infty .$ If not, then
there exists a subsequence $\{ x_{n_i} \}_{i\in \mathbb{N} }$  of $\{ x_n \}_{n\in \mathbb{N} }$ such that $\{ x_{n_j} \}_{j\in \mathbb{N} }$  converges weakly to some $q$ in $C$ with $ p \neq q.$ In view of Proposition  \ref{Bregman alpha02} again, we conclude that $q \in F(T).$ By Lemma \ref{Bregman alpha Ishi}, $ \displaystyle \lim_{n\to \infty} D_f(x_n ,z ) $ exists for all $z \in F(T).$ Thus we obtain by the Bregman Opial-like property that 
$$\aligned
\displaystyle \lim_{n\to \infty} D_f(x_n , p )&= \displaystyle \lim_{i\to \infty} D_f(x_{n_i} , p )< \displaystyle \lim_{i\to \infty} D_f(x_{n_i} , q )\\
&= \displaystyle \lim_{n\to \infty} D_f(x_{n} , q )= \displaystyle \lim_{j\to \infty} D_f(x_{n_j} , q )\\
&< \displaystyle \lim_{j\to \infty} D_f(x_{n_j} , p )= \displaystyle \lim_{n\to \infty} D_f(x_{n} , p ).\\
\endaligned $$ 
This is a contradiction. Thus we have $p = q,$ and the desired assertion follows.
\end{proof}
\begin{thm} \label{Bregman alpha Ishi04}
Let $f:E\to \mathbb{R}$ be a strictly convex, {G$\hat{a}$teaux differentiable function,} bounded on bounded sets and uniformly convex on bounded subsets of $E.$ Let $C$ be a nonempty, closed and convex subset of a reflexive Banach space $E.$ Let $T:C\to C $ a Bregman generalized $ \alpha$-nonexpansive and Bregman skew quasi-nonexpansive mapping. Let $\{\beta_n\}_{n\in \mathbb{N} }$ and $\{\gamma_n\}_{n\in \mathbb{N} }$  be sequences in $[0,1).$ Then $\{ x_n \}_{n\in \mathbb{N} }$ be a sequence with $ x_1 \in C$ defined by the Ishikawa iteration \eqref{Ishikawa}. Then $\{ x_n \}_{n\in \mathbb{N} }$ converges strongly to a fixed point $z$ of $T.$ 
\end{thm}
\begin{proof}
By Corollary \ref{Bregman alpha04}, we see that the fixed point set $F(T)$ of $T$ is nonempty. In view of Theorem \ref{Bregman alpha Ishi02}, we obtain that $\{ x_n \}_{n\in \mathbb{N} }$ is bounded and $ \displaystyle \liminf_{n\to \infty} \|Tx_{n}-x_{n} \|=0. $ By the compactness of $C$, there exists a subsequence $\{ x_{n_k} \}_{k\in \mathbb{N} }$ of $\{ x_n \}_{n\in \mathbb{N} }$ such that $\{ x_{n_k} \}_{k\in \mathbb{N} }$  converges strongly to some $z$ in $C$. In view of Lemma \ref{Naraghirad02} we deduce that $ \displaystyle \lim_{k\to \infty} D_f(x_{n_k} ,z )= 0 $. We 
can even assume that $ \displaystyle \lim_{k\to \infty} \|Tx_{n_k}-x_{n_k} \|=0 $, and in particular, $\{ Tx_{n_k} \}_{k\in \mathbb{N} }$ is bounded.
Since $ \nabla f $ is uniformly norm-to-norm continuous on bounded subsets of $E$ (see, for example, \cite{Zalinescu}),
$$ \displaystyle \lim_{k\to \infty} \| \nabla f(Tx_{n_k})-\nabla f (x_{n_k})\| = 0. $$ 
Let $  M_3 = sup \{ \| \nabla f (x_{n_k}) \| , \| Tx_{n_k} \| , \| \nabla f (z) \|, \| \nabla f (Tz) \| : k \in \mathbb{N} \} < +\infty .$ In view of Lemma \ref{Bregman alpha01}, we obtain \vskip 2mm
$\aligned
D_{f}(x_{n_k}, Tz)&\leq D_{f}(x_{n_k},Tx_{n_k})+ (1 - \alpha ) D_{f}(x_{n_k},z) + \alpha D_{f}(Tx_{n_k}, Tz) \\
&\quad+ \alpha \langle x_{n_k} -Tx_{n_k},\nabla f(z) - \nabla f(Tz) \rangle + \langle x_{n_k} - Tx_{n_k} ,\nabla f(Tx_{n_k}) - \nabla f(Tz) \rangle  \\
&=  D_{f}(x_{n_k},Tx_{n_k})  + (1- \alpha )D_{f}(x_{n_k},z)\\
&\quad+ \alpha[ D_{f}(Tx_{n_k},x_{n_k}) + D_f (x_{n_k},Tz)+ \langle Tx_{n_k} -x_{n_k} ,\nabla f (x_{n_k}) - \nabla f (Tz)\rangle]\\
&\quad+ \alpha \langle x_{n_k} -Tx_{n_k},\nabla f(z) - \nabla f(Tz) \rangle + \langle x_{n_k} - Tx_{n_k} ,\nabla f(Tx_{n_k}) - \nabla f(Tz) \rangle  \\
&=  D_{f}(x_{n_k},Tx_{n_k})  + (1- \alpha )D_{f}(x_{n_k},z)\\
&\quad+ \alpha D_ {f}(Tx_{n_k},x_{n_k}) + \alpha D_f (x_{n_k},Tz)+ \alpha \langle Tx_{n_k} -x_{n_k} ,\nabla f (x_{n_k}) - \nabla f (Tz)\rangle\\
&\quad+ \alpha \langle x_{n_k} -Tx_{n_k},\nabla f(z) - \nabla f(Tz) \rangle + \langle x_{n_k} - Tx_{n_k} ,\nabla f(Tx_{n_k}) - \nabla f(Tz) \rangle  \\
&=  D_{f}(x_{n_k},Tx_{n_k})  + (1- \alpha )D_{f}(x_{n_k},z)\\
&\quad -\alpha D_ {f}(x_{n_k},Tx_{n_k}) + \alpha \langle x_{n_k} - Tx_{n_k}, \nabla f (x_{n_k})- \nabla f (Tx_{n_k})\rangle \\
&\quad+ \alpha D_f (x_{n_k},Tz)+ \alpha \langle x_{n_k} - Tx_{n_k} ,\nabla f (Tz) - \nabla f (x_{n_k})\rangle\\
&\quad+ \alpha \langle x_{n_k} -Tx_{n_k},\nabla f(z) - \nabla f(Tz) \rangle + \langle x_{n_k} - Tx_{n_k} ,\nabla f(Tx_{n_k}) - \nabla f(Tz) \rangle  \\
&= (1- \alpha )D_f (x_{n_k},Tx_{n_k})+(1- \alpha )D_f (x_{n_k},z)+ \alpha D_f (x_{n_k},T_z)\\
&\quad+ \alpha \langle x_{n_k} - Tx_{n_k},\nabla f(z) - \nabla f(Tx_{n_k}) \rangle + \langle x_{n_k} - Tx_{n_k} ,\nabla f(Tx_{n_k}) - \nabla f(Tz) \rangle  \\
&\leq (1- \alpha )D_f (x_{n_k},Tx_{n_k})+(1- \alpha )D_f (x_{n_k},z)+ \alpha D_f (x_{n_k},T_z)\\
&\quad+ \alpha\| x_{n_k} - Tx_{n_k} \| \|\nabla f(z)-\nabla f(Tx_{n_k}) \| + \| x_{n_k} - Tx_{n_k} \| \|\nabla f(Tx_{n_k}) - \nabla f(Tz) \|  \\
&\leq (1- \alpha )D_f (x_{n_k},Tx_{n_k})+(1- \alpha )D_f (x_{n_k},z)+ \alpha D_f (x_{n_k},T_z)\\
&\quad+ 2 \alpha M_3 \| x_{n_k} - Tx_{n_k} \| + 2M_3 \| x_{n_k} - Tx_{n_k} \| \\
&\leq (1- \alpha )D_f (x_{n_k},Tx_{n_k})+ D_f (x_{n_k},z)\\
&\quad+ 2 \alpha M_3 \| x_{n_k} - Tx_{n_k} \| + 2M_3 \| x_{n_k} - Tx_{n_k} \| \\
\endaligned $ \vskip 2mm
for all $k \in \mathbb{N}.$
It follows $ \displaystyle \lim_{k\to \infty} \|x_{n_k}-Tz \|=0. $ Thus we have $Tz = z.$ In view of Lemmas \ref{Bregman alpha Ishi} and
\ref{Naraghirad02}, we conclude that $ \displaystyle \lim_{n\to \infty} \|x_{n}- z \|=0. $ Therefore, $z$ is the strong limit of the sequence $\{ x_n \}_{n\in \mathbb{N} }.$
\end{proof}

\section{Bregman Noor's type iteration for Bregman generalized $\alpha $-nonexpansive mappings }

We propose the following Bregman Noor's type iteration. Let E be a reflexive Banach space and let $C$ be a nonempty, closed and convex subset of $E.$ let $f:E\to \mathbb{R}$ be a strictly convex and \emph{G$\hat{a}$teaux differentiable function.} Let $T:C\to C $  be a Bregman generalized $ \alpha$-nonexpansive mapping such that the fixed point set $F(T)$ is nonempty. Let $\{ x_n \}_{n\in \mathbb{N} },$ $\{ y_n\}_{n\in \mathbb{N} }$ and $\{ z_n\}_{n\in \mathbb{N} }$ be three sequences defined by

		\begin{equation} \label{Bregman Noor}
\left \{\begin{array}{l} 

z_n = \alpha_n \nabla f (Tx_n) + (1 - \alpha_n ) \nabla f (x_n),\\
y_n = \nabla f^{*} [\beta_n \nabla f (Tz_n) + ( 1 - \beta_n )\nabla f (x_n)],\\
x_{n+1}  = proj_C^f \big( \nabla f^{*} [\gamma_n \nabla f (Ty_n) + ( 1 - \gamma_n) \nabla f (x_n) ] \big), 
\end{array}\right.
		\end{equation}
		
where $\{ \alpha_n \}_{n\in \mathbb{N} }$, $\{ \beta_n\}_{n\in \mathbb{N} }$ and $\{ \gamma_n\}_{n\in \mathbb{N} }$ are arbitrary sequences in $[0,1).$

\begin{lem} \label{Bregman Noor01}
Let $f:E\to \mathbb{R}$ be a strongly coercive Bregman function. Let $C$ be a nonempty, closed and convex subset of a reflexive Banach space $E.$ Let $T:C\to C $ be a Bregman quasi-nonexpansive mapping. Let $\{x_n\}_{n\in \mathbb{N} }$, $\{y_n\}_{n\in \mathbb{N} }$ and $\{z_n\}_{n\in \mathbb{N} }$ be three sequences defined by \eqref{Bregman Noor} such that $\{\alpha_n\}_{n\in \mathbb{N} }$ , $\{\beta_n\}_{n\in \mathbb{N} }$ and $\{\gamma_n\}_{n\in \mathbb{N} }$ are arbitrary sequences in $[0,1).$ Then the following assertions hold:\\

(1)   $ max \{ D_f(w,x_{n+1}), D_f (w,y_n), D_f (w,z_n)\} \leq D_f (w,x_n) $ for all $ w $ in $ F(T)$ and $ n=1,2,...$.\vskip 2mm
(2)   $ \displaystyle \lim_{n\to \infty} D_f(w,x_n)$ exists for any $w$ in $F(T).$ \vskip 2mm 
\end{lem}
\begin{proof}
 Let $w$ in $F(T).$ In view of Lemma \ref{Butnariu02} and \eqref{Bregman Noor}, we conclude that
 \begin{center}
$\aligned
 D_f (w,z_n) &= D_f \big(w,\alpha_n \nabla f (Tx_n) + (1-\alpha_n )\nabla f(x_n)\big) \\
&= V\big(w,\alpha_n \nabla f (Tx_n) + (1-\alpha_n )\nabla f(x_n)\big)\\
&\leq \alpha_n V\big(w, \nabla f (Tx_n)\big) + (1-\alpha_n )V \big(w,\nabla f(x_n)\big)\\
&= \alpha_n D_f \big(w,Tx_n\big) + (1-\alpha_n )D_f \big(w,x_n\big)\\ 
&\leq \alpha_n D_f \big(w,x_n\big) + (1-\alpha_n )D_f \big(w,x_n\big)\\ 
&= D_f \big(w,x_n\big). \\
 \endaligned $
\end{center}
Also,
 \begin{center}
$\aligned
 D_f (w,y_n) &= D_f \big(w,\nabla f^{*} [\beta_n \nabla f (Tz_n) + (1-\beta_n )\nabla f(x_n)\big) \\
&= V\big(w,\beta_n \nabla f (Tz_n) + (1-\beta_n )\nabla f(x_n)\big)\\
&\leq \beta_n V\big(w, \nabla f (Tz_n)\big) + (1-\beta_n )V \big(w,\nabla f(x_n)\big)\\
&= \beta_n D_f \big(w,Tz_n\big) + (1-\beta_n )D_f \big(w,x_n\big)\\ 
&\leq \beta_n D_f \big(w,z_n\big) + (1-\beta_n )D_f \big(w,x_n\big)\\
&= \beta_n D_f \big(w,x_n\big) + (1-\beta_n )D_f \big(w,x_n\big)\\ 
&= D_f \big(w,x_n\big). \\
 \endaligned $
\end{center}
Consequently, using \eqref{Bregman projection02} we have
 \begin{center}
$\aligned
D_f (w,x_{n+1}) &= D_f \big(w,proj_C^f \big(\nabla f^{*} [\gamma_n \nabla f (Ty_n) + ( 1 - \gamma_n) \nabla f (x_n) ])\big) \\
&\leq D_f \big(w,\nabla f^{*} [\gamma_n \nabla f (Ty_n) + ( 1 - \gamma_n) \nabla f (x_n) ]\big) \\
&= V\big(w,\gamma_n \nabla f (Ty_n) + ( 1 - \gamma_n) \nabla f (x_n)\big)\\
&\leq \gamma_n V\big(w, \nabla f (Ty_n)\big) + (1-\gamma_n )V \big(w,\nabla f(x_n)\big)\\
&= \gamma_n D_f \big(w,Ty_n\big) + (1-\gamma_n )D_f \big(w,x_n\big)\\ 
&\leq \gamma_n D_f \big(w,y_n\big) + (1-\gamma_n )D_f \big(w,x_n\big)\\
&= \gamma_n D_f \big(w,x_n\big) + (1-\gamma_n )D_f \big(w,x_n\big)\\ 
&= D_f \big(w,x_n\big). \\
 \endaligned $
\end{center}
This implies that $\{D_f (w,x_n)\}_{n\in \mathbb{N} }$  is a bounded and nonincreasing sequence for all $w$ in $F(T).$ Thus we have $ \displaystyle \lim_{n\to \infty} D_f(w,x_n)$  exists for any $w$ in $F(T).$
\end{proof}

\begin{thm} \label{Bregman Noor02}
Let $f:E\to \mathbb{R}$ be a strongly coercive Bregman function which is bounded on bounded sets, locally uniformly convex and locally uniformly smooth on $E.$ Let $C$ be a nonempty, closed and convex subset of a reflexive Banach space $E.$ Let $T:C\to C $ a Bregman generalized $ \alpha$-nonexpansive mapping. Let $\{\alpha_n\}_{n\in \mathbb{N} },$ $\{\beta_n\}_{n\in \mathbb{N} }$ and $\{\gamma_n\}_{n\in \mathbb{N} }$  be sequences in $[0,1)$ satisfying the control condition:
\begin{equation} \label{condition}
\displaystyle \sum_{n=1}^\infty \gamma_n \beta_n \alpha_n (1-\alpha_n) = +\infty.
\end{equation}

Let $\{ x_n \}_{n\in \mathbb{N} }$ be a sequence generated by the algorithm \eqref{Bregman Noor}. Then the following are equivalent.  \\

(1)  There exists a bounded sequence $\{ x_n \}_{n\in \mathbb{N} } \subset C $ such that  $ \displaystyle \liminf_{n\to \infty} \|Tx_{n}-x_{n} \|=0. $ \\

(2)  The fixed point set $F(T)\neq \emptyset .$
\end{thm}
\begin{proof}
 The implication $(1)\Longrightarrow(2)$ follows similarly as in the first part of the proof of Theorem \ref{Bregman alpha Ishi02}. For the implication $(2)\Longrightarrow (1)$, we assume $F(T)\neq \emptyset $. The boundedness of the sequences $\{ x_n \}_{n\in \mathbb{N} }$, $\{ y_n \}_{n\in \mathbb{N} }$ and $\{ z_n \}_{n\in \mathbb{N} }$ follows from Lemma \ref{Bregman Noor01} and Definition \ref{Kohsaka01}. Since T is a Bregman quasi-nonexpansive mapping, for any $q$ in $F(T)$ we have
\begin{center}
$  D_f (q,Tx_n) \leq D_f (q, x_n), \,\,\,\,\,\,\ \forall n \in \mathbb{N}.$
\end{center}
This, together with Definition \ref{Kohsaka01} and the boundedness of $\{ x_n \}_{n\in \mathbb{N} },$ implies that $\{ Tx_n \}_{n\in \mathbb{N} }$ is bounded.

The function $f$ is bounded on bounded subsets of $E$ and therefore $\nabla f $ is also bounded on bounded subsets of $E^{*}$ (see, for example, [\cite{Butnariu01}, Proposition 1.1.11] for more details). This implies the sequences $\{ \nabla f (x_n) \}_{n\in \mathbb{N} },$ $\{ \nabla f (y_n) \}_{n\in \mathbb{N} },$ $\{ \nabla f (z_n) \}_{n\in \mathbb{N} },$ $\{ \nabla f (Tz_n) \}_{n\in \mathbb{N} } $, $\{ \nabla f (Ty_n) \}_{n\in \mathbb{N} }$ and $\{ \nabla f (Tx_n) \}_{n\in \mathbb{N} }$ are bounded in $E^{*}.$

In view of Proposition \ref{Zalinescu02}, we have that $ dom \,\,f^{*} = E^{*} $ and $ f^{*} $ is strongly coercive and uniformly convex on bounded subsets of $E^{*}.$ Let $ s_2 = sup \{\| \nabla f (x_n) \|, \| \nabla f (Tx_n) \| : n \in \mathbb{N}\} < \infty $ and let $ \rho_{s_2}^{*} : E^{*}\rightarrow \mathbb{R} $ be the gauge of uniform convexity of the (Fenchel) conjugate function $ f^{*} .$\vskip 2mm
\textbf{Claim.} For any $p \in F(T)$ and $n \in \mathbb{N},$
\begin{equation} \label{function star}
 D_f (p,z_n)\leq D_f (p,x_n) - \alpha_n (1 - \alpha_n )\rho_{s_2}^{*} ( \| \nabla f (x_n) - \nabla f (Tx_n ) \| ).
\end{equation}

Let $p \in F(T).$ For each $n \in \mathbb{N},$ it follows from the definition of Bregman distance (\ref{Bregman}), Lemma \ref{Butnariu02}, (\ref{locally}) and (\ref{Bregman Noor}) that \vskip 2mm
$\aligned
D_f (p,z_n) &= f(p) -f(z_n) - \langle p - z_n , \nabla f (z_n ) \rangle \\
&= f(p) + f^{*} \big(\nabla f (z_n)\big) - \langle z_n , \nabla f (z_n)\rangle - \langle p - z_n , \nabla f (z_n ) \rangle \\
&= f(p) + f^{*} \big(\nabla f (z_n)\big) - \langle z_n , \nabla f (z_n)\rangle - \langle p, \nabla f (z_n)\rangle + \langle z_n , \nabla f (z_n ) \rangle \\
&= f(p) + f^{*} \big( (1-\alpha_n )\nabla f (x_n) + \alpha_n \nabla f (Tx_n)\big) - \langle p, ((1-\alpha_n)\nabla f (x_n)+ \alpha_n \nabla f (Tx_n ) \rangle \\
&\leq (1-\alpha_n)f(p) + \alpha_n f(p) + (1-\alpha_n) f^{*}(\nabla f (x_n) + \alpha_n f^{*} (\nabla f (Tx_n))\\
&\quad - \alpha_n (1- \alpha_n ) \rho_{s_2}^{*}( \| \nabla f (x_n) - \nabla f (Tx_n) \| )-(1-\alpha_n) \langle p, \nabla f (x_n)\rangle - \alpha_n \langle p,\nabla f (Tx_n)\rangle \\
&= (1-\alpha_n) [f(p) + f^{*}(\nabla f (x_n))- \langle p, \nabla f (x_n)\rangle ]\\
&\quad + \alpha_n [ f(p)+ f^{*}(\nabla f (Tx_n))-\langle p, \nabla f (Tx_n)\rangle ] - \alpha_n (1- \alpha_n ) \rho_{s_2}^{*}( \| \nabla f (x_n) - \nabla f (Tx_n) \| )\\
&= (1-\alpha_n) [f(p) - f (x_n)+\langle x_n ,\nabla f (x_n)\rangle - \langle p, \nabla f (x_n)\rangle ]\\
&\quad + \alpha_n [ f(p)- f (Tx_n) +\langle Tx_n , \nabla f (Tx_n)\rangle -\langle p, \nabla f (Tx_n)\rangle ]\\
&\quad - \alpha_n (1- \alpha_n ) \rho_{s_2}^{*}( \| \nabla f (x_n) - \nabla f (Tx_n) \| )\\
&= (1-\alpha_n) D_f ( p,x_n ) + \alpha_n D_f (p,Tx_n) - \alpha_n (1-\alpha_n)\rho_{s_2}^{*}( \| \nabla f (x_n) - \nabla f (Tx_n) \| )\\
&\leq (1-\alpha_n) D_f ( p,x_n ) + \alpha_n D_f (p,x_n) - \alpha_n (1-\alpha_n)\rho_{s_2}^{*}( \| \nabla f (x_n) - \nabla f (Tx_n) \| )\\
&= D_f ( p,x_n ) - \alpha_n (1-\alpha_n)\rho_{s_2}^{*}( \| \nabla f (x_n) - \nabla f (Tx_n) \| ).\\
\endaligned $\vskip 2mm
In view of Lemma \ref{Butnariu02} and (\ref{function star}), we obtain
$$\aligned
 D_f (p,y_n) &= D_f \big(p,\beta_n \nabla f (Tz_n) + (1-\beta_n )\nabla f(x_n)\big) \\
&= V\big(p,\beta_n \nabla f (Tz_n) + (1-\beta_n )\nabla f(x_n)\big)\\
&\leq \beta_n V\big(p, \nabla f (Tz_n)\big) + (1-\beta_n )V \big(p,\nabla f(x_n)\big)\\
&= \beta_n D_f \big(p,Tz_n\big) + (1-\beta_n )D_f \big(p,x_n\big)\\ 
&\leq \beta_n D_f \big(p,z_n\big) + (1-\beta_n )D_f \big(p,x_n\big)\\
&= \beta_n D_f \big(p,x_n\big) - \beta_n \alpha_n (1-\alpha_n)\rho_{s_2}^{*}( \| \nabla f (x_n) - \nabla f (Tx_n) \| ).\\ 
 \endaligned $$ \vskip 2mm
In view of \ref{Butnariu02} and (\ref{function star}), we obtain
$$\aligned
D_f (p,x_{n+1}) &= D_f \big(p,\nabla f^{*} [\gamma_n \nabla f (Ty_n) + ( 1 - \gamma_n) \nabla f (x_n) ])\big) \\
&= V\big(p,\gamma_n \nabla f (Ty_n) + ( 1 - \gamma_n) \nabla f (x_n)\big)\\
&\leq \gamma_n V\big(p, \nabla f (Ty_n)\big) + (1-\gamma_n )V \big(p,\nabla f(x_n)\big)\\
&= \gamma_n D_f \big(p,Ty_n\big) + (1-\gamma_n )D_f \big(p,x_n\big)\\ 
&\leq \gamma_n D_f \big(p,y_n\big) + (1-\gamma_n )D_f \big(p,x_n\big)\\
&= \gamma_n D_f \big(p,x_n\big) - \gamma_n \alpha_n \beta_n (1-\alpha_n)\rho_{s_2}^{*}( \| \nabla f (x_n) - \nabla f (Tx_n) \| )+ (1-\gamma_n )D_f \big(p,x_n\big) \\ 
&\leq D_f \big(p,x_n\big)- \gamma_n \alpha_n \beta_n (1-\alpha_n)\rho_{s_2}^{*}( \| \nabla f (x_n) - \nabla f (Tx_n) \| ).  \\
 \endaligned $$ 
Thus we have
\begin{equation} 
\gamma_n \alpha_n \beta_n (1-\alpha_n)\rho_{s_2}^{*}( \| \nabla f (x_n) - \nabla f (Tx_n) \| )\leq D_f (p,x_n)- D_f (p,x_{n+1})
\end{equation}
Since $\{D_f (x_n,z)\}_{n\in \mathbb{N} }$ converges, together with the control condition (\ref{condition}), we have
$$ \displaystyle \lim_{n\to \infty} \| \nabla f (x_n)- \nabla f (Tx_n) \| =0 . $$
Since $ \nabla f^{*} $ is uniformly norm-to-norm continuous on bounded subsets of $E^{*}$ (see, for example,\cite{Zalinescu}), \\ 
we arrive at
\begin{equation}
\displaystyle \liminf_{n\to \infty} \| x_n- Tx_n \|=0.
\end{equation}
\end{proof}

\begin{thm} \label{Bregman Noor03}
Let $f:E\to \mathbb{R}$ be a strongly coercive Bregman function which is bounded on bounded sets, locally uniformly convex and locally uniformly smooth on $E.$ Let $C$ be a nonempty, closed and convex subset of a reflexive Banach space $E.$ Let $T:C\to C $ a Bregman generalized $ \alpha$-nonexpansive mapping with $F(T)\neq \emptyset.$ Let $\{\alpha_n\}_{n\in \mathbb{N} },$ $\{\beta_n\}_{n\in \mathbb{N} }$ and $\{\gamma_n\}_{n\in \mathbb{N} }$  be three sequences in $[0,1)$ satisfying the control conditions $ \displaystyle \Sigma_{n=1}^\infty \gamma_n \beta_n \alpha_n (1-\alpha_n) = +\infty. $ Let $\{x_n\}_{n\in \mathbb{N} }$ be a  generated by the algorithm \eqref{Bregman Noor}. Then, there exists a subsequence $\{x_{n_i}\}_{i \in \mathbb{N} }$ of $\{x_n\}_{n\in \mathbb{N} }$ which converges weakly to a fixed point of $T$ as $ i\rightarrow \infty. $
\end{thm}

\begin{proof}
It follows from Theorem \ref{Bregman Noor02} that $\{x_n\}_{n\in \mathbb{N} }$ is bounded and $ \displaystyle \liminf_{n\to \infty} \| Tx_n-x_n \|=0.$ Since $E$ is reflexive, then there exists a subsequence $\{x_{n_i}\}_{i \in \mathbb{N} }$ of $\{x_n\}_{n\in \mathbb{N} }$ such that $ x_{n_i} \rightharpoonup p \in C $ as $ i \rightarrow \infty .$ \\ In view of Proposition \ref{Bregman alpha02}, we conclude that $ p \in F(T) $ and the desired conclusion follows.
\end{proof}

The construction of fixed points of nonexpansive mappings via Halpern's algorithm  \cite{Halpern} has been extensively investigated recently in the current literature (see, for example, \cite{Reich01} and the references therein). Numerous results have been proved on Halpern's iterations for
nonexpansive mappings in Hilbert and Banach spaces (see, e.g.,  \cite{Nilsrakoo,Suzuki,Takahashi02}).

Before dealing with the strong convergence of a Halpern-type iterative algorithm, we need the following lemmas \ref{Mainge} and \ref{Xu}.

\begin{thm} \label{Bregman Noor04}
Let $f:E\to \mathbb{R}$ be a strongly coercive Bregman function which is bounded on bounded sets, locally uniformly convex and locally uniformly smooth on $E.$ Let $C$ be a nonempty, closed and convex subset of a reflexive Banach space $E.$ Let $T:C\to C $ a Bregman generalized $ \alpha$-nonexpansive mapping with $F(T)\neq \emptyset .$ Let $\{\alpha_n\}_{n\in \mathbb{N} },$ $\{\beta_n\}_{n\in \mathbb{N} }$ and $\{\gamma_n\}_{n\in \mathbb{N} }$  be three sequences in $[0,1)$ satisfying the control conditions : \vskip 2mm
(a) $ \displaystyle \lim_{n\to \infty} \gamma_n = 0. $ \vskip 2mm
(b) $ \displaystyle \Sigma_{n=1}^\infty \gamma_n = +\infty; $ \vskip 2mm
(c) $ 0 < \displaystyle \liminf_{n\to \infty} \beta_n \leq \displaystyle \limsup_{n\to \infty} \beta_n < 1. $ \vskip 2mm 

Let $\{x_n\}_{n\in \mathbb{N} }$ be a sequence generated by 
		\begin{equation} \label{Bregman Halpern}
\left \{\begin{array}{l} 

u \in C, x_1 \in C \text{chosen arbitrarily},\\
z_n = \alpha_n \nabla f (x_n) + (1 - \alpha_n ) \nabla f (Tx_n),\\
y_n = \nabla f^{*} [\beta_n \nabla f (x_n) + ( 1 - \beta_n )\nabla f (z_n)],\\
x_{n+1}  = proj_C^f \big( \nabla f^{*} [\gamma_n \nabla f (u) + ( 1 - \gamma_n) \nabla f (y_n) ] \big),\,\,\ \text{for} \,\,\
n \in \mathbb{N}. 
\end{array}\right.
		\end{equation}
Then the sequence $\{x_{n_i}\}_{i \in \mathbb{N} }$ defined in \eqref{Bregman Halpern} converges strongly to $ proj_{F(T)}^f u $ as $ n \rightarrow \infty .$ 
\end{thm}
\begin{proof}
We divide the proof into several steps. In view of Lemma \ref{Hussain}, we conclude that $F(T)$ is closed and convex. Set
$$ w = proj_{F(T)}^f u. $$
\textbf{Step 1.} We prove that $\{x_n\}_{n\in \mathbb{N} }$, $\{y_n\}_{n\in \mathbb{N} }$ and $\{z_n\}_{n\in \mathbb{N} }$ are bounded sequences in $C.$\vskip 2mm
We first show that $\{x_n\}_{n\in \mathbb{N} }$ is bounded. Let $p \in F(T)  $ be fixed. In view of Lemma \ref{Butnariu02} and \eqref{Bregman Halpern}, we have
$$\aligned
 D_f (p,z_n) &= D_f \big(p,\alpha_n \nabla f (x_n) + (1-\alpha_n )\nabla f(Tx_n)\big) \\
&= V\big(p,\alpha_n \nabla f (x_n) + (1-\alpha_n )\nabla f(Tx_n)\big)\\
&\leq \alpha_n V\big(p, \nabla f (x_n)\big) + (1-\alpha_n )V \big(p,\nabla f(Tx_n)\big)\\
&= \alpha_n D_f \big(p,x_n\big) + (1-\alpha_n )D_f \big(p,Tx_n\big)\\ 
&\leq \alpha_n D_f \big(p,x_n\big) + (1-\alpha_n )D_f \big(p,x_n\big)\\
&= D_f \big(p,x_n\big).\\ 
 \endaligned $$ \vskip 2mm
Also,
$$\aligned
 D_f (p,y_n) &= D_f \big(p,\nabla f^{*} [\beta_n \nabla f (x_n) + (1-\beta_n )\nabla f(z_n)] \big) \\
&= V\big(p,\beta_n \nabla f (x_n) + (1-\beta_n )\nabla f(z_n)\big)\\
&\leq \beta_n V\big(p, \nabla f (x_n)\big) + (1-\beta_n )V \big(p,\nabla f(z_n)\big)\\
&= \beta_n D_f \big(p,x_n\big) + (1-\beta_n )D_f \big(p,z_n\big)\\ 
&\leq \beta_n D_f \big(p,x_n\big) + (1-\beta_n )D_f \big(p,x_n\big)\\
&= D_f \big(p,x_n\big).\\ 
 \endaligned $$ \vskip 2mm
 This, together with (\ref{Bregman Noor}), implies that
$$\aligned
 D_f (p,x_{n+1}) &= D_f \big(p,proj_C^{f} \big (\nabla f^{*} [\gamma_n \nabla f (u) + (1-\gamma_n )\nabla f(y_n)] \big) \\
&= D_f \big(p, \nabla f^{*} [\gamma_n \nabla f (u) + (1-\gamma_n )\nabla f(y_n)]\big) \\
&= V\big(p,\gamma_n \nabla f (u) + (1-\gamma_n )\nabla f(y_n)\big)\\
&\leq \gamma_n V\big(p, \nabla f (u)\big) + (1-\gamma_n )V \big(p,\nabla f(y_n)\big)\\
&= \gamma_n D_f \big(p,u\big) + (1-\gamma_n )D_f \big(p,y_n\big)\\ 
&\leq \gamma_n D_f \big(p,u\big) + (1-\gamma_n )D_f \big(p,y_n\big)\\
&\leq \gamma_n D_f \big(p,u\big) + (1-\gamma_n )D_f \big(p,x_n\big)\\
&\leq max \{ D_f \big(p,u\big), D_f \big(p,x_n\big)\}.\\ 
 \endaligned $$ \vskip 2mm
By induction, we obtain
\begin{equation} \label{Bregman Halpern01}
D_f \big(p, x_{n+1} \big) \leq max \{ D_f \big(p,u\big), D_f \big(p,x_1\big)\}, \,\,\,\ \forall n \in \mathbb{N}. 
\end{equation}

Let $ M_4 = max \{ D_f \big(p,u\big), D_f \big(p,x_n\big): n \in \mathbb{N} \} .$
It follows from (\ref{Bregman Halpern01}) that the sequence $\{D_f (p,x_n)\}_{n\in \mathbb{N} }$  is bounded and hence there exists $ M_4 > 0 $ such that
\begin{equation} \label{Bregman Halpern02}
D_f \big(p, x_{n} \big) \leq M_4, \,\,\,\ \forall n \in \mathbb{N}. 
\end{equation}
In view of Definition \ref{Kohsaka01}, we deduce that the sequence $\{x_n\}_{n\in \mathbb{N} }$, is bounded. Since $T$ is a
Bregman quasi-nonexpansive mapping from $C$ into itself, we conclude that
\begin{equation} \label{Bregman Halpern03}
D_f \big(p, Tx_{n} \big) \leq D_f \big(p, x_{n} \big), \,\,\,\ \forall n \in \mathbb{N}. 
\end{equation}

This, together with Definition \ref{Kohsaka01} and the boundedness of $\{x_n\}_{n\in \mathbb{N} }$, implies that $\{Tx_n\}_{n\in \mathbb{N} }$ is
bounded. The function $f$ is bounded on bounded subsets of $E$ and therefore $ \nabla f $ is also bounded on bounded subsets of $E^{*}$ (see, for example, [\cite{Butnariu01}, Proposition 1.1.11] for more details). This, together with Step 1, implies that the sequences $\{ \nabla f (x_n)\}_{n\in \mathbb{N} }$, $\{ \nabla f (y_n)\}_{n\in \mathbb{N} }$, $\{ \nabla f (z_n)\}_{n\in \mathbb{N} }$ and $\{ \nabla f (Tx_n)\}_{n\in \mathbb{N} }$ are bounded in $ E^{*}$. In view of Proposition \ref{Zalinescu02}, we obtain that $dom\,\, f^{*} = E^{*}$ and $ f^{*}$ is strongly coercive and uniformly convex on bounded subsets of $E.$ Let $ s_3 = sup \{\| \nabla f (x_n) \|, \| \nabla f (Tx_n) \| : n \in \mathbb{N}\} $ and let $ \rho_{s_3}^{*} : E^{*}\rightarrow \mathbb{R} $ be the gauge of uniform convexity of the (Fenchel) conjugate function $ f^{*}.$\vskip 2mm
\textbf{Step 2.} We prove that 
\begin{equation} \label{Bregman Halpern04}
D_f (w,z_n)\leq D_f (w,x_n) - \alpha_n (1- \alpha_n )(1-\beta_n) \rho_{s_3}^{*} \big ( \| \nabla f (x_n) - \nabla f (Tx_n) \| \big), \,\,\,\ \forall n \in \mathbb{N}.
\end{equation} 
For each $n$ in $ \mathbb{N}$, in view of the definition of Bregman distance (\ref{Bregman}), Lemma \ref{Butnariu02} and (\ref{conjugate}),
we obtain \vskip 2mm
$\aligned
D_f (w,z_n) &= f(w) -f(z_n) - \langle w - z_n , \nabla f (z_n ) \rangle \\
&= f(w) + f^{*} \big(\nabla f (z_n)\big) - \langle z_n , \nabla f (z_n)\rangle - \langle w - z_n , \nabla f (z_n ) \rangle \\
&= f(w) + f^{*} \big(\nabla f (z_n)\big) - \langle z_n , \nabla f (z_n)\rangle - \langle w, \nabla f (z_n)\rangle + \langle z_n , \nabla f (z_n ) \rangle \\
&= f(w) + f^{*} \big( (1-\alpha_n )\nabla f (x_n) + \alpha_n \nabla f (Tx_n)\big) - \langle w, ((1-\alpha_n)\nabla f (x_n)+ \alpha_n \nabla f (Tx_n ) \rangle \\
&\leq (1-\alpha_n)f(w) + \alpha_n f(w) + (1-\alpha_n) f^{*}(\nabla f (x_n) + \alpha_n f^{*} (\nabla f (Tx_n))\\
&\quad - \alpha_n (1- \alpha_n ) \rho_{s_3}^{*}( \| \nabla f (x_n) - \nabla f (Tx_n) \| )\\
&\quad -(1-\alpha_n) \langle w, \nabla f (x_n)\rangle - \alpha_n \langle w,\nabla f (Tx_n)\rangle \\
&= (1-\alpha_n) [f(w) + f^{*}(\nabla f (x_n))- \langle w, \nabla f (x_n)\rangle ]\\
&\quad + \alpha_n [ f(w)+ f^{*}(\nabla f (Tx_n))-\langle w, \nabla f (Tx_n)\rangle ]\\
&\quad - \alpha_n (1- \alpha_n ) \rho_{s_3}^{*}( \| \nabla f (x_n) - \nabla f (Tx_n) \| )\\
&= (1-\alpha_n) [f(w) - f (x_n)+\langle x_n ,\nabla f (x_n)\rangle - \langle w, \nabla f (x_n)\rangle ]\\
&\quad + \alpha_n [ f(w)- f (Tx_n) +\langle Tx_n , \nabla f (Tx_n)\rangle -\langle w, \nabla f (Tx_n)\rangle ]\\
&\quad - \alpha_n (1- \alpha_n ) \rho_{s_3}^{*}( \| \nabla f (x_n) - \nabla f (Tx_n) \| )\\
&= (1-\alpha_n) D_f ( w,x_n ) + \alpha_n D_f (w,Tx_n) - \alpha_n (1-\alpha_n)\rho_{s_3}^{*}( \| \nabla f (x_n) - \nabla f (Tx_n) \| )\\
&\leq (1-\alpha_n) D_f ( w,x_n ) + \alpha_n D_f (w,x_n) - \alpha_n (1-\alpha_n)\rho_{s_3}^{*}( \| \nabla f (x_n) - \nabla f (Tx_n) \| )\\
&= D_f ( w,x_n ) - \alpha_n (1-\alpha_n)\rho_{s_3}^{*}( \| \nabla f (x_n) - \nabla f (Tx_n) \| ).\\
\endaligned $ \vskip 2mm
Also,
$$\aligned
 D_f (w,y_n) &= D_f \big(w,\beta_n \nabla f (x_n) + (1-\beta_n )\nabla f(z_n)\big) \\
&= V\big(w,\beta_n \nabla f (x_n) + (1-\beta_n )\nabla f(z_n)\big)\\
&\leq \beta_n V\big(w, \nabla f (x_n)\big) + (1-\beta_n )V \big(w,\nabla f(z_n)\big)\\
&= \beta_n D_f \big(w,x_n\big) + (1-\beta_n )D_f \big(w,z_n\big)\\ 
&\leq \beta_n D_f \big(w,x_n\big) + (1-\beta_n )D_f \big(w,x_n\big)\\
&\quad- \alpha_n (1-\alpha_n)(1- \beta_n )\rho_{s_3}^{*}( \| \nabla f (x_n) - \nabla f (Tx_n) \| )\\
&= D_f \big(w,x_n\big) - \alpha_n (1-\alpha_n)(1- \beta_n )\rho_{s_3}^{*}( \| \nabla f (x_n) - \nabla f (Tx_n) \| ).\\ 
 \endaligned $$ \vskip 2mm
In view of Lemma \ref{Butnariu02} and (\ref{Bregman Halpern04}), we obtain
\begin{equation} \label{Bregman Halpern05}
\aligned
 D_f (w,x_{n+1}) &= D_f \big(w,proj_C^{f} \big (\nabla f^{*} [\gamma_n \nabla f (u) + (1-\gamma_n )\nabla f(y_n)] \big) \\
&= D_f \big(w, \nabla f^{*} [\gamma_n \nabla f (u) + (1-\gamma_n )\nabla f(y_n)]\big) \\
&= V\big(w,\gamma_n \nabla f (u) + (1-\gamma_n )\nabla f(y_n)\big)\\
&\leq \gamma_n V\big(w, \nabla f (u)\big) + (1-\gamma_n )V \big(w,\nabla f(y_n)\big)\\
&= \gamma_n D_f \big(w,u\big) + (1-\gamma_n )D_f \big(w,y_n\big)\\ 
&\leq \gamma_n D_f \big(w,u\big) \\
&\quad + (1-\gamma_n )[D_f \big(w,x_n\big) - \alpha_n (1-\alpha_n)(1- \beta_n )\rho_{s_3}^{*}( \| \nabla f (x_n) - \nabla f (Tx_n) \| )].\\
 \endaligned 
\end{equation}
Let $ M_5 = sup \{ | D_f \big(w,u\big)- D_f \big(w,x_n\big)| +  \alpha_n (1-\alpha_n)(1- \beta_n )\rho_{s_3}^{*}( \| \nabla f (x_n) - \nabla f (Tx_n) \| ): n \in \mathbb{N} \}.$\vskip 2mm
It follows from (\ref{Bregman Halpern05}) that 
\begin{equation} \label{Bregman Halpern06}
\alpha_n (1-\alpha_n)(1- \beta_n )\rho_{s_3}^{*}( \| \nabla f (x_n) - \nabla f (Tx_n) \| )\leq D_f (w,x_n) - D_f(w,x_{n+1}) + \gamma_n M_5.
\end{equation}
Let
$$ w_n = \nabla f^{*} [\gamma_n \nabla f (u) + (1-\gamma_n )\nabla f(y_n)]. $$
Then $ x_{n+1} = proj_C^{f} (w_n), \,\,\, \forall n \in \mathbb{N}.$ In view of Lemma \ref{Butnariu02} and (\ref{Bregman Halpern04}) we obtain
\begin{equation} \label{Bregman Halpern07}
\aligned
 D_f (w,x_{n+1}) &= D_f \big(w,proj_C^{f} \big (\nabla f^{*} [\gamma_n \nabla f (u) + (1-\gamma_n )\nabla f(y_n)] \big) \\
&\leq D_f \big(w, \nabla f^{*} [\gamma_n \nabla f (u) + (1-\gamma_n )\nabla f(y_n)]\big) \\
&= V\big(w,\gamma_n \nabla f (u) + (1-\gamma_n )\nabla f(y_n)\big)\\
&\leq V\big(w,\gamma_n \nabla f (u) + (1-\gamma_n )\nabla f(y_n)\big) -\gamma_n (\nabla f (u)- \nabla f (w))\\
&\quad- \langle \nabla f^{*} [\gamma_n \nabla f (u) + (1-\gamma_n )\nabla f(y_n)]- w, -\gamma_n (\nabla f (u)- \nabla f (w))\rangle\\
&= V\big(w,\gamma_n \nabla f (w) + (1-\gamma_n )\nabla f(y_n)\big)+ \gamma_n \langle w_n - w , \nabla f (u)- \nabla f (w) \rangle \\
&\leq \gamma_n V\big(w, \nabla f (w)\big) + (1-\gamma_n )V \big(w,\nabla f(y_n)\big)+ \gamma_n \langle w_n - w , \nabla f (u)- \nabla f (w) \rangle \\
&= \gamma_n D_f \big(w,w\big) + (1-\gamma_n )D_f \big(w,y_n\big) + \gamma_n \langle w_n - w , \nabla f (u)- \nabla f (w) \rangle \\ 
&= (1-\gamma_n )D_f \big(w,y_n\big) + \gamma_n \langle w_n - w , \nabla f (u)- \nabla f (w) \rangle. \\ 
 \endaligned 
\end{equation}
\textbf{Step 3.} We show that $ x_n \rightarrow w $ as $ n \rightarrow \infty.$ 

\textit{Case 1.} If there exists $ n_0 \in \mathbb{N} $  such that $\{ D_f (w,x_n)\}_{n={n_0}}^{\infty} $  is nonincreasing, then $\{ D_f (w,x_n)\}_{n\in \mathbb{N} }$
is convergent. Thus, we have $ D_f (w,x_n)- D_f (w,x_{n+1})\rightarrow 0 $ as $ n \rightarrow \infty $. This, together with (\ref{Bregman Halpern06}) and conditions (a) and (c), implies that 
$$ \displaystyle \lim_{n\to \infty} \rho_{s_3}^{*} \big ( \| \nabla f (x_n) - \nabla f (Tx_n) \| = 0.$$
Therefore, from the property of $ \rho_{s_3}^{*} $
we deduce that
\begin{equation} \label{Bregman Halpern08}
\displaystyle \lim_{n\to \infty} \| \nabla f (x_n) - \nabla f (Tx_n) \| = 0.
\end{equation}
Since $ \nabla f^{*} = (\nabla f)^{-1}$ (Lemma \ref{Zalinescu}) is uniformly norm-to-norm continuous on bounded subsets
of $ E^{*} $ (see, for example,  \cite{Zalinescu}), we arrive at
\begin{equation} \label{Bregman Halpern09}
\displaystyle \lim_{n\to \infty} \| x_n - Tx_n \| = 0.
\end{equation}
On the other hand, we have
$$ \aligned
 D_f (Tx_n,z_{n}) &= D_f \big(Tx_n, \gamma_n \nabla f (x_n) + (1-\gamma_n )\nabla f(Tx_n)\big) \\
 &= V\big(Tx_n,\gamma_n \nabla f (x_n) + (1-\gamma_n )\nabla f(Tx_n)\big)\\
&\leq \gamma_n V\big(Tx_n, \nabla f (x_n)\big) + (1-\gamma_n )V \big(Tx_n,\nabla f(Tx_n)\big)\\
&= \gamma_n D_f \big(Tx_n,x_n\big) + (1-\gamma_n )D_f \big(Tx_n,Tx_n\big)\\ 
&\leq \gamma_n D_f \big(Tx_n,x_n\big). \\
 \endaligned  $$ 
This, together with Lemma \ref{Naraghirad02} and (\ref{Bregman Halpern09}), implies that
$$ \displaystyle \lim_{n\to \infty} D_f (Tx_n,z_n)= 0.  $$
Similarly, we have
$$ D_f (z_n , w_n ) \leq \gamma_n D_f (z_n,u ) + (1 - \gamma_n) D_f (z_n,z_n ) = \gamma_n D_f ( z_n , u) \rightarrow 0 \,\,\,\,  \text{as} \,\,\, n \rightarrow \infty. $$
In view of Lemma \ref{Naraghirad02} and (\ref{Bregman Halpern09}), we conclude that
$$ \displaystyle \lim_{n\to \infty} \| z_n - Tx_n \| = 0 \,\,\, \text{and} \,\,\ \displaystyle \lim_{n\to \infty} \| w_n - x_n \| = 0.  $$
Since $\{x_n\}_{n\in \mathbb{N} }$ is bounded, together with (\ref{Bregman projection}) we can assume there exists a subsequence
$\{x_{n_i}\}_{i \in \mathbb{N} }$ of $\{x_n\}_{n\in \mathbb{N} }$ such that $ x_{n_i} \rightharpoonup z \in F(T) $ (Proposition \ref{Bregman alpha02}) and
$$ \aligned
\displaystyle \limsup_{n\to \infty} \langle x_n - w, \nabla f (u) - \nabla f (w)\rangle &= \displaystyle \lim_{i\to \infty} \langle x_{n_i} - w, \nabla f (u) - \nabla f (w)\rangle \\
&= \langle y - w, \nabla f (u) - \nabla f (w)\rangle \leq 0 . \\
\endaligned  $$ 
We thus conclude
$$ \displaystyle \limsup_{n\to \infty} \langle z_n - w, \nabla f (u) - \nabla f (w)\rangle = \displaystyle \limsup_{n\to \infty} \langle x_{n} - w, \nabla f (u) - \nabla f (w)\rangle \leq 0. $$
The desired result follows from Lemma \ref{Naraghirad02} and \ref{Xu} and (\ref{Bregman Halpern07}).\\

\textit{Case 2.} Suppose there exists a subsequence $\{n_{i}\}_{i\in \mathbb{N} }$ of $\{n\}_{n\in \mathbb{N} }$  such that
$$ D_f (w,x_{n_i}) < D_f (w, x_{n_i+1} ), \,\,\, \forall i \in \mathbb{N}. $$ By Lemma \ref{Mainge}, there exists a non-decreasing sequence $\{m_{k}\}_{k\in \mathbb{N} }$ of positive integers such that $ m_k \rightarrow \infty, $
$$ D_f (w,x_{m_k}) < D_f (w, x_{m_k+1}) \,\,\, \text{and} \,\,\, D_f (w,x_{k}) < D_f (w, x_{m_k+1}), \,\,\, \forall k \in \mathbb{N}. $$
This, together with (\ref{Bregman Halpern06}), implies that
$$ \alpha_{m_k} (1-\alpha_{m_k})(1- \beta_{m_k} )\rho_{s_3}^{*}( \| \nabla f (x_{m_k}) - \nabla f (Tx_{m_k}) \| )\leq D_f (w,x_{m_k}) - D_f(w,x_{m_k+1}) + \gamma_{m_k} M_5 \leq \gamma_{m_k} M_5, $$
$ \forall k \in \mathbb{N}. $\\
Then, by conditions (a) and (c), we get
$$ \displaystyle \lim_{k\to \infty} \rho_{s_3}^{*} \big ( \| \nabla g (x_{m_k}) - \nabla f (Tx_{m_k}) \| = 0.$$
By the same argument, as in Case 1, we arrive at
\begin{equation} \label{Bregman Halpern10}
\displaystyle \limsup_{k\to \infty} \langle w_{m_k} - w, \nabla f (u) - \nabla f (w)\rangle = \displaystyle \limsup_{k\to \infty} \langle x_{m_k} - w, \nabla f (u) - \nabla f (w)\rangle \leq 0.
\end{equation}
It follows from (\ref{Bregman Halpern07}) that
\begin{equation} \label{Bregman Halpern11}
D_f (w,x_{m_k+1}) \leq (1-\gamma_{m_k}) D_f (w, x_{m_k}) + \gamma_{m_k} D_f (w,x_{m_k}) + \gamma_{m_k} \langle z_{m_k} - w, \nabla f (u) - \nabla f (w)\rangle. 
\end{equation}
Since $ D_f (w, x_{m_k})\leq D_f (w,x_{m_k+1}),$ we have that 
$$ \aligned
 \gamma_{m_k}D_f (w,x_{m_k}) &\leq D_f (w, x_{m_k}) - D_f (w, x_{m_k+1}) + \gamma_{m_k} \langle w_{m_k} - w, \nabla f (u) - \nabla f (w)\rangle\\
 &\leq  \gamma_{m_k} \langle w_{m_k} - w, \nabla f (u) - \nabla f (w)\rangle. \\
 \endaligned  $$ 
 In particular, since $\gamma_{m_k} > 0, $ we obtain 
$$ D_f (w,x_{m_k}) \leq\langle w_{m_k} - w, \nabla f (u) - \nabla f (w)\rangle.  $$
In view of (\ref{Bregman Halpern10}), we deduce that
$$ \displaystyle \lim_{k\to \infty} D_f (w,x_{m_k})= 0. $$
This, together with (\ref{Bregman Halpern11}), implies
$$ \displaystyle \lim_{k\to \infty} D_f (w,x_{m_k+1})= 0 . $$
On the other hand, we have $ D_f (w,x_{k}) \leq D_f (w,x_{m_k+1}), \,\,\, \forall k \in \mathbb{N} $.\\  This ensures that $ x_k \rightarrow w \,\,\, \text{as} \,\,\, k \rightarrow \infty $ by Lemma \ref{Naraghirad02}. 
\end{proof}


\section{Numerical example}

In this section we discuss the direct application of Theorem \ref{Bregman Noor04} on a typical example on a real line. 
\begin{exmp}
Let $ E=\mathbb{R} ,$ the set of all real numbers, $ C=[-1,1],$ and let $ f : \mathbb{R}\to \mathbb{R} $ be defined by $ f(x)=\frac{4}{5}x^2.$ Let $ T:C \to C $ be defined by $ Tx=\frac{1}{5}x .$ Setting $ \lbrace \alpha_n \rbrace =\lbrace \frac{n+1}{4n} \rbrace, \lbrace \beta_n \rbrace = \lbrace \frac{n+1}{5n} \rbrace,\\ \lbrace \gamma_n \rbrace = \lbrace \frac{1}{500n}\rbrace, \forall n \geq 1.$ Consider the following: 
\begin{equation*}
\aligned
&E=\mathbb{R},\quad C=[-1,1],\quad Tx=\frac{1}{5}x,\quad f(x)=\frac{4}{5}x^2,\quad\nabla f(x)=\frac{8}{5}x,\\
&f^*(x^*)=\sup\{\langle x^*,x\rangle-f(x) : x\in E\},\quad f^*(z)=\frac{5}{16}z^2,\quad\nabla f^*(z)=\frac{5}{8}z,\\
&\alpha_n=\frac{n+1}{4n},\quad\beta_n=\frac{n+1}{5n},\quad\gamma_n=\frac{1}{500n},\\
& z_n = \alpha_n \nabla f (x_n) + (1 - \alpha_n ) \nabla f (Tx_n)=\frac{16n+8}{25n}x_n,\\
& y_n = \nabla f^{*} [\beta_n \nabla f (x_n) + ( 1 - \beta_n )\nabla f (z_n)]=\frac{n+1}{5n}x_n+\frac{4n-1}{5n}z_n,\\
& x_{n+1} = \nabla f^{*} [\gamma_n \nabla f (u) + ( 1 - \gamma_n) \nabla f (y_n) ]=\frac{u}{500n}+\frac{500n-1}{500n}y_n.
\endaligned
\end{equation*}

Given initial values $x_1=-0.8$ and $u=0.1.$  Using the software Matlab 2017b, we have the following Figure \ref{fig.1} and Table \ref{table.1} which show that $\{x_n\}, \{z_n\}$ and $\{y_n\}$ converge to $w = \{0\}$ as $n\rightarrow\infty.$
\end{exmp}

\begin{figure}[h]
	\centering
	\includegraphics[{width=13cm}]{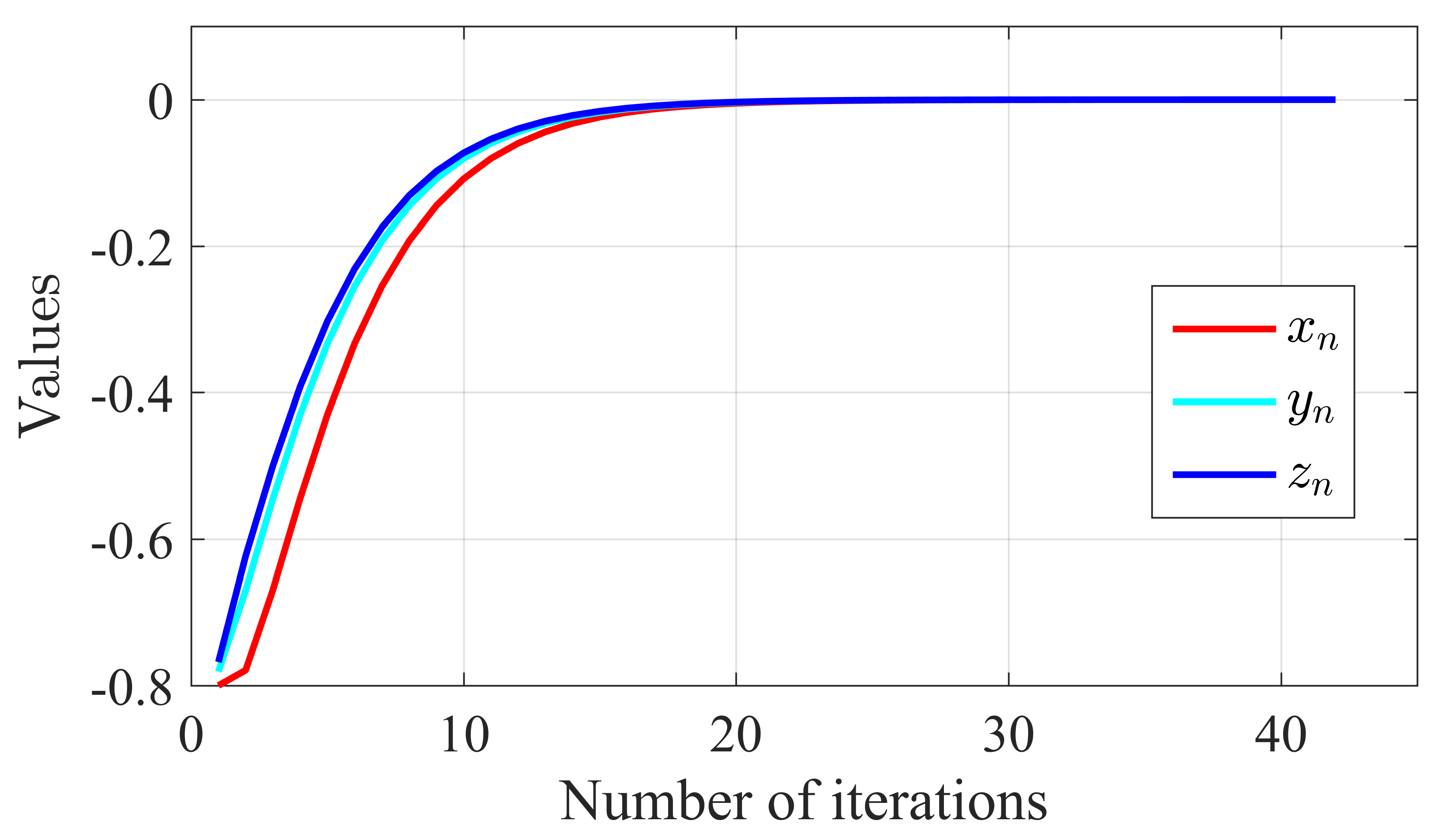}
	\caption{Plotting of $\{x_n\}, \{y_n\}$ and $\{z_n\}$ converge to $w = \{0\}$ as $n\rightarrow\infty$  }\label{fig.1}
\end{figure}  

\begin{table}[h]
    \centering
    \begin{tabular}{l l l l l }
        \hline
No. of iterations & $ z_n $ & $ y_n $ & $ x_n $ & $ \Vert x_{n+1} - x_{n} \Vert$ \\ \hline
1&	-0.7680000&	-0.7808000&	-0.7790384&	0.0209616\\
2&	-0.6232307&	-0.6699730&	-0.6692031&	0.1098353\\
3&	-0.4996716&	-0.5448800&	-0.5444501&	0.1247530\\
4&	-0.3920041&	-0.4301156&	-0.4298505&	0.1145996\\
5&	-0.3026148&	-0.3331513&	-0.3329781&	0.0968724\\
6&	-0.2308648&	-0.2546912&	-0.2545730&	0.0784051\\
7&	-0.1745643&	-0.1928520&	-0.1927684&	0.0618046\\
8&	-0.1310825&	-0.1449618&	-0.1449006&	0.0478678\\
9&	-0.0978884&	-0.1083355&	-0.1082892&	0.0366113\\
10&	-0.0727704&	-0.0805845&	-0.0805484&	0.0277408\\
11&	-0.0538942&	-0.0597097&	-0.0596806&	0.0208678\\
12&	-0.0397871&	-0.0440974&	-0.0440733&	0.0156073\\
13&	-0.0292918&	-0.0324755&	-0.0324551&	0.0116182\\
14&	-0.0215131&	-0.0238578&	-0.0238402&	0.0086150\\
15&	-0.0157663&	-0.0174887&	-0.0174730&	0.0063671\\
16&	-0.0115322&	-0.0127946&	-0.0127805&	0.0046925\\
17&	-0.0084201&	-0.0093435&	-0.0093306&	0.0034499\\
18&	-0.0061375&	-0.0068116&	-0.0067997&	0.0025309\\
19&	-0.0044663&	-0.0049576&	-0.0049465&	0.0018532\\
20&	-0.0032449&	-0.0036023&	-0.0035919&	0.0013546\\
21&	-0.0023536&	-0.0026130&	-0.0026032&	0.0009887\\
22&	-0.0017039&	-0.0018920&	-0.0018827&	0.0007205\\
23&	-0.0012311&	-0.0013671&	-0.0013583&	0.0005244\\
24&	-0.0008874&	-0.0009855&	-0.0009771&	0.0003812\\
25&	-0.0006379&	-0.0007084&	-0.0007004&	0.0002767\\
26&	-0.0004569&	-0.0005074&	-0.0004997&	0.0002007\\
27&	-0.0003257&	-0.0003618&	-0.0003544&	0.0001453\\
28&	-0.0002309&	-0.0002564&	-0.0002493&	0.0001051\\
29&	-0.0001623&	-0.0001803&	-0.0001734&	0.0000759\\
30&	-0.0001128&	-0.0001253&	-0.0001187&	0.0000547\\
31&	-0.0000772&	-0.0000857&	-0.0000793&	0.0000394\\
32&	-0.0000515&	-0.0000572&	-0.0000510&	0.0000283\\
33&	-0.0000331&	-0.0000368&	-0.0000307&	0.0000202\\
34&	-0.0000200&	-0.0000222&	-0.0000163&	0.0000144\\
35&	-0.0000106&	-0.0000118&	-0.0000060&	0.0000103\\
36&	-0.0000039&	-0.0000044&	0.0000012&	0.0000072\\
37&	0.0000008&	0.0000009&	0.0000063&	0.0000051\\
38&	0.0000041&	0.0000045&	0.0000098&	0.0000035\\
39&	0.0000063&	0.0000070&	0.0000122&	0.0000024\\
40&	0.0000079&	0.0000088&	0.0000138&	0.0000016\\
41&	0.0000089&	0.0000099&	0.0000148&	0.0000010\\
42&	0.0000096&	0.0000106&	0.0000154&	0.0000006 \\ \hline
\end{tabular} \vskip 3 mm
\caption{Values of $ z_n $, $ y_n $ and $ x_n $  }\label{table.1}
\end{table}


\newpage
\section{Acknowledgements}
This project was supported by the Theoretical and Computational Science (TaCS) Center under Computational and Applied Science for Smart Innovation Research Cluster (CLASSIC), Faculty of Science, KMUTT (For CLASSIC-Cluster and TaCS-Center).

\bibliographystyle {amsplain}

\end{document}